\documentclass[10pt]{amsart}
\usepackage{color}
\usepackage{epsfig}
\usepackage{psfrag, graphics}
\usepackage{amsthm}
\usepackage{amssymb}
\usepackage{amsmath}
\usepackage{amscd}
\usepackage{graphicx}
\usepackage{epstopdf}
\newcommand{\eqdef}{\;{:=}\;}

\newtheorem {Theorem}   {Theorem}

\numberwithin{Theorem}{section}

\newtheorem {Lemma}[Theorem]    {Lemma}
\newtheorem {Proposition}[Theorem]{Proposition}
\newtheorem {Corollary}[Theorem]{Corollary}

\theoremstyle{definition}

\theoremstyle{remark}
\newtheorem{Remark}[Theorem]{Remark}
\newtheorem{Example}[Theorem]{Example}

\expandafter\chardef\csname pre amssym.def
at\endcsname=\the\catcode`\@ \catcode`\@=11
\def\undefine#1{\let#1\undefined}
\def\newsymbol#1#2#3#4#5{\let\next@\relax
 \ifnum#2=\@ne\let\next@\msafam@\else
 \ifnum#2=\tw@\let\next@\msbfam@\fi\fi
 \mathchardef#1="#3\next@#4#5}
\def\mathhexbox@#1#2#3{\relax
 \ifmmode\mathpalette{}{\m@th\mathchar"#1#2#3}%
 \else\leavevmode\hbox{$\m@th\mathchar"#1#2#3$}\fi}
\def\hexnumber@#1{\ifcase#1 0\or 1\or 2\or 3\or 4\or 5\or 6\or 7\or 8\or
 9\or A\or B\or C\or D\or E\or F\fi}

\font\teneufm=eufm10 \font\seveneufm=eufm7 \font\fiveeufm=eufm5
\newfam\eufmfam
\textfont\eufmfam=\teneufm \scriptfont\eufmfam=\seveneufm
\scriptscriptfont\eufmfam=\fiveeufm

\catcode`\@=\csname pre amssym.def at\endcsname

\newcommand{\AC}{{\mathcal A}}

\newcommand{\DD}{{\mathcal D}}
\newcommand{\EE}{{\mathcal E}}

\newcommand{\GG}{{\mathcal G}}
\newcommand{\HH}{{\mathcal H}}

\newcommand{\JJ}{{\mathcal J}}

\newcommand{\LL}{{\mathcal L}}

\newcommand{\NN}{{\mathcal N}}

\newcommand{\PP}{{\mathcal P}}

\newcommand{\RR}{{\mathcal R}}
\newcommand{\SC}{{\mathcal S}}

\newcommand{\VV}{{\mathcal V}}

\def    \C      {{\mathbb C}}
\def    \R      {{\mathbb R}}
\def    \Z      {{\mathbb Z}}

\newcommand{\Ham}{{\mathit Ham}}
\newcommand{\id}{{\operatorname{id}}}

\def    \om       {\omega}
\def    \eps      {\epsilon}
\def    \ssminus  {\smallsetminus}
\def    \p        {\partial}

\def    \length   {\operatorname{length}}

\def    \Crit      {\operatorname{Crit}}

\def    \Hh     {\mathbf{H}}

\def    \12      {\frac{1}{2}}
\def    \Vol      {\operatorname{Vol}}

\def    \hl       {H_{\scriptscriptstyle{L}}}

\def    \gl       {G_{\scriptscriptstyle{L}}}

\def    \mor      {\operatorname{ind}}

\def    \H   {\operatorname{H}}  
\def    \CM   {\operatorname{CM}} 
\def    \czi   {\operatorname{\mu_{\scriptscriptstyle{CZ}}^{int}}} 
\def    \cz   {\operatorname{\mu_{\scriptscriptstyle{CZ}}}} 
\def    \mas   {\operatorname{\mu_{\scriptscriptstyle{Maslov}}}} 
\def    \masl   {\operatorname{\mu^{\scriptscriptstyle{L}}_{\scriptscriptstyle{Maslov}}}}
\def    \mor   {\operatorname{I_{\scriptscriptstyle{Morse}}}} 

\begin{document}







\title[Maslov class rigidity via Hofer's geometry]{Maslov class rigidity for Lagrangian submanifolds via Hofer's geometry}

\author[Ely Kerman]{Ely Kerman}
\address{Department of Mathematics, University of Illinois at Urbana-Champaign, Urbana, IL 61801, USA }
\email{ekerman@math.uiuc.edu, sirikci@uiuc.edu }
\author[Nil \.{I}.  \c{S}irik\c{c}i]{Nil \.{I}.  \c{S}irik\c{c}i}

\date{\today}

\thanks{This research was partially supported by a grant from the Campus Research
Board of the University of Illinois at Urbana-Champaign.}


\bigskip

\begin{abstract}
In this work, we establish new rigidity results for the Maslov class
of  Lagrangian submanifolds in large classes of closed and convex symplectic manifolds. 
Our main result establishes upper bounds for the minimal Maslov number 
of displaceable Lagrangian submanifolds which are product manifolds 
whose factors each admit a metric of negative sectional curvature. Such Lagrangian submanifolds exist in every symplectic manifold of dimension greater than six or equal to four.

The proof utilizes the relations between closed geodesics on the Lagrangian, the periodic orbits of geometric Hamiltonian flows supported near the Lagrangian, and the length minimizing properties of these flows
with respect to the negative Hofer length functional.


\end{abstract}

\maketitle

\section{Introduction}

Among the fundamental rigidity phenomena in symplectic topology are  restrictions on Lagrangian submanifolds  that are undetectable by topological methods. In this work we establish such restrictions which are expressible in terms of the Maslov class. All the  Lagrangian submanifolds we consider are assumed to be connected, compact and without boundary.

The first \emph{symplectic} restrictions on Lagrangian submanifolds were discovered  in \cite{gr}, where Gromov  proves that there are no exact Lagrangian submanifolds of $\R^{2n}$ equipped with its standard symplectic form $\om_{2n}$. This result can be rephrased as the fact that the symplectic area class 
\begin{equation*}
\label{ }
\om_{2n}^L\colon \pi_2(\R^{2n},L) \to \R,
\end{equation*}
which is defined by integrating $\om_{2n}$ over smooth representatives, is nontrivial for any Lagrangian submanifold $L$ of $(\R^{2n}, \om_{2n})$.

Another type of restriction on Lagrangian submanifolds involves the 
Maslov class.
Recall that for a Lagrangian submanifold $L$ of a symplectic manifold $(M, \om)$, the Maslov class  is a homomorphism
\begin{equation*}
\label{ }
\masl \colon \pi_2(M,L) \to \Z,
\end{equation*}
that measures the \emph{winding} of $TL$ in $TM$ along loops in $L$.\footnote{The precise definition is given in \S \ref{maslov}.}  The minimal Maslov number of $L$ is the smallest nonnegative integer $N_L$ such that 
$\masl(\pi_2(M,L)) = N_L \Z.$ As an application of his proof of the Weinstein conjecture for $(\R^{2n}, \om_{2n})$,
 Viterbo established the first restrictions on the Maslov class in the following result.

\begin{Theorem}(\cite{vi})
\label{vi}
Let $L$ be a closed Lagrangian submanifold of $(\R^{2n}, \om_{2n})$. If $L$ is a 
torus, then $N_L$ is in $[2, n+1].$ If $L$ admits a metric with negative sectional curvature
then $N_L =2$.
\end{Theorem}

A Lagrangian submanifold $L \subset (M,\om)$ is called  {\bf monotone} if the symplectic 
area class and the Maslov class are proportional with a nonnegative constant of proportionality, i.e., $\om^L = \lambda \masl$ for some $\lambda \geq 0$. 
In this case, restrictions on the two classes are related and one can relax the assumptions on the intrinsic 
geometry of $L$
needed in Theorem \ref{vi}. The first  rigidity results for the Maslov class of monotone Lagrangians  were obtained for $(\R^{2n}, \om_{2n})$  by Polterovich in \cite{po1}. Later, in \cite{oh}, Oh introduced a spectral sequence for  Lagrangian Floer theory and used it to establish similar rigidity results for displaceable monotone Lagrangian submanifolds in more general ambient symplectic manifolds.\footnote{Note that a symplectic manifold which admits a monotone Lagrangian submanifold is itself monotone in the sense of \cite{fl2}.}
Indeed, monotone Lagrangian submanifolds  are the natural setting for the construction of Lagrangian Floer homology. This homology  vanishes if the Lagrangian submanifold is displaceable, and its grading is determined by the Maslov class. Oh observed that if the Lagrangian Floer homology of $L \subset (M,\om)$ exists, e.g., if $N_L \geq 2$, then it can not vanish unless $N_L$ is less than $\12 \dim M +2$.
This strategy to obtain rigidity results for the Maslov class has been generalized, refined and exploited in
many recent works, see for example \cite{al, bi, bci, bco, bu, fooo}.  

The primary goal of the present work is to obtain rigidity results for the Maslov class away from settings where Lagrangian Floer homology can be used. In particular, our assumptions on the symplectic area class and the Maslov class only involve their values on $\pi_2(M)$ rather than $\pi_2(M,L)$. We also avoid the use of holomorphic discs with boundary on the Lagrangian submanifold. Hence, we do not need to manage or avoid codimension one bubbling phenomena. Instead of Lagrangian Floer theory, our argument utilizes two results from  the study of Hofer's length functional for Hamiltonian paths. The first of these results, due to Sikorav, is the fact that the Hamiltonian flow generated by an autonomous Hamiltonian with displaceable support, does not minimize the Hofer length for all time. In particular, Sikorav's result implies that the flow of an autonomous 
Hamiltonian which is supported near a displaceable Lagrangian submanifold is, eventually, not length minimizing. The second
result from Hofer's geometry involves the following theme: if a Hamiltonian path  is not length minimizing then it has a contractible periodic orbit with additional \emph{special} properties, see for example \cite{ho2,ke1,kl,lm, mcd, mcsl}. Here, we use the version of this phenomena established in \cite{ke1} which states that if a Hamiltonian flow is not length minimizing then it has contractible periodic orbits with  spanning discs such that the corresponding  action values and Conley-Zehnder indices lie in certain intervals. By applying these results to reparameterizations of perturbed cogeodesic flows that are supported near a displaceable Lagrangian submanifold,  we detect \emph{nonconstant} periodic orbits with spanning discs for which the 
corresponding Conley-Zehnder index  is $\12 \dim M$. Utilizing an identity
which relates this Conley-Zehnder index 
to the Morse index of the corresponding perturbed geodesic and the Maslov index  of the spanning disc, we obtain  the rigidity results for the Maslov class stated below.

\subsection{Statement and discussion of results}
\label{statement}

Throughout this work, we will denote our ambient symplectic manifold by $(M,\om)$ and its dimension will be $2n$. Lagrangian submanifolds will be denoted by $L$ and will always be assumed to be connected, compact and without boundary.

We will consider symplectic manifolds which are either closed or convex and have the following two additional properties. We assume that  the  {\bf index of rationality} of $(M,\om)$,
$$
r(M,\om) = \inf_{A \in \pi_2(M)} \left\{ \om(A) \mid \om(A)>0\right\},
$$
is positive.\footnote{Our convention here is that the infimum over the empty set is equal to infinity.} Such a manifold is said to be {\bf rational}. We will also assume that there
is a (possibly negative) constant $\upsilon$ such that $c_1(A) = \upsilon \om(A)$
for every $A$ in $\pi_2(M)$.  A symplectic manifold with this property will be called {\bf proportional}.

Recall  that $L$ is  {\bf displaceable} if there is a Hamiltonian diffeomorphism $\phi$ of $(M, \om)$ such that $\phi(L) \cap L = \emptyset.$ 
For example, for any symplectic manifold $(M,  \om)$, every Lagrangian submanifold of the product $(M \times \R^2, \om \oplus \om_2)$,  is displaceable.\footnote{Note that if $(M,\om)$ is compact, then  $(M \times \R^2, \om \oplus \om_2)$ is convex. This basic example is a primary motivation for our consideration of the
class of  convex symplectic manifolds.}  A Lagrangian submanifold  of a rational symplectic manifold $(M,\om)$ is said to be {\bf easily displaceable} if it is displaced by a Hamiltonian diffeomorphism $\phi^1_H$ that is the time one flow of a Hamiltonian $H$ whose Hofer norm, $$\|H\| = \int_0^1 \max_{p\in M}H(t,p) \, dt - \int_0^1 \min_{p\in M}H(t,p) \, dt,$$
is less than $r(M,\om)/2$.
In other words,  $L$ is easily displaceable if its displacement energy, 
$$
e(L,M,\om)= \inf_{H} \left\{ \|H\| \mid \phi^1_H(L) \cap L = \emptyset \right\},
$$
is less than $r(M, \om)/2$.

If $(M, \om)$ is weakly exact, i.e., $\om|_{\pi_2(M)} = 0$, then the index of rationality is infinite and the notions of \emph{displaceable} and  \emph{easily displaceable} Lagrangian submanifolds are equivalent. The simplest example of a symplectic manifold which is not weakly exact is the two-sphere with an area form. In this case, every Lagrangian submanifold is either monotone and nondisplaceable or is easily displaceable. The following example shows that being easily displaceable is actually a stronger condition than being displaceable.
\begin{Example}
Let $\vartheta$ be an area form on $S^2$ with total area equal to one. Consider the symplectic manifold
$$
(W,\om_j)=(S^2 \times S^2, \vartheta \oplus (1+\alpha_j) \vartheta)
$$
where $\alpha_j$ is an increasing sequence of positive rational numbers which converges to an irrational number $\alpha$.
Let $\hat{L} \subset S^2$ be an embedded circle which divides $S^2$ into two regions of unequal $\vartheta$-area.
For every $j$,  the Lagrangian submanifold $L =\hat{L}  \times \hat{L}$ of $(W,  \om_j)$ is displaceable.
We will show that the  the indices of rationality $r(W,  \om_j)$ converge to zero, whereas the displacement energy of $L$ in $(W,  \om_j)$ is uniformly bounded away from zero. Hence, for sufficiently large values of $j$
the Lagrangian submanifold $L \subset (W,  \om_j)$ is displaceable but not easily displaceable.

For all $j$ we have $0 < r(W,\om_j) \leq \alpha_j < \alpha$. Passing to a subsequence, if necessary, 
we may therefore assume that $\lim_{j \to \infty} r(W, \om_j)$ exists. Since $\om_{\infty} \eqdef \vartheta \oplus (1+\alpha)\vartheta
$, and $\alpha$ is irrational, the limit $r(W,  \om_{\infty})$ is zero.

To obtain the uniform lower bound on the displacement energies $e(L, W, \om_j)$, we use a deep thereom from \cite{ch}. Assume that the symplectic manifold $(M,\om)$ admits an $\om$-compatible almost complex structure $J$ such that the metric $\om(\cdot, J, \cdot)$ is  complete. Denote by $\SC$, the space of nonconstant $J$-holomorphic spheres in $M$, and let $\DD$ be the set of nonconstant $J$-holomorphic discs in $M$ with boundary on $L$.
The  quantities 
\begin{equation*}
\label{ }
\hbar(M,\om, J) = \min_{u \in \SC} \int_{S^2} u^* \om
\end{equation*}
and
\begin{equation*}
\label{ }
\hbar_D(M,L,\om,J) = \min_{v \in \DD} \int_{D^2} v^* \om
\end{equation*}
are well-defined and positive. In \cite{ch}, Chekanov proves  that 
\begin{equation}
\label{chek}
e(L,M,\om) \geq  \min\{\hbar(M,\om, J) , \hbar_D(M,L,\om,J) \}.
\end{equation}

Now choose a $\vartheta$-compatible almost complex structure $\hat{J}$ on $S^2$ and set 
$J =\hat{J} \oplus \hat{J}$. Note that $J$ is $\om_j$-compatible for each $j$ and the spaces $\SC$ and 
$\DD$ do not depend on $j$. Hence,  
$$
\lim_{j\to \infty} \hbar(W, \om_j, J) = \hbar(W, \om_{\infty}, J) >0
$$
and 
$$
\lim_{j\to \infty} \hbar_D(W, L,\om_j, J) =\hbar_D(W, L, \om_{\infty},J) >0.
$$
It then follows from inequality \eqref{chek} that the displacement energies $e(L,M,\om_j)$ are bounded uniformly away from zero.
\end{Example}

As in Theorem \ref{vi}, we need to make some assumptions on the metrics which 
our Lagrangian submanifolds admit.  A manifold  $L$ is said to be {\bf split hyperbolic} if it is diffeomorphic to a product manifold  
\begin{equation*}
\label{}
L= P_1 \times \dots \times P_k,
\end{equation*}
such that each of the factors $P_j$ admits a metric with negative sectional curvature. Our convention will be to label the factors of $L$ so that $\dim P_j \leq \dim P_{j+1}.$

Note that the set of split hyperbolic manifolds is strictly larger than the set of manifolds which admit metrics of negative sectional curvature. This follows from  a theorem of Priessmann \cite{pr}, which states that no (nontrivial) product manifold  admits a metric of negative sectional curvature. More importantly, 
Lagrangian submanifolds which are both displaceable and split hyperbolic, are somewhat ubiquitous. 
\begin{Lemma}
\label{examples}
There  is an (easily) displaceable, split hyperbolic Lagrangian submanifold 
in every (rational) symplectic manifold with dimension four or with dimension 
greater than six.
\end{Lemma}

\begin{proof}
The product of a Lagrangian embedding into $\R^{2m}$ and a Lagrangian immersion into $\R^{2n}$ is homotopic to a Lagrangian embedding
into $\R^{2(m+n)}$  (see,  \cite{aula}, Proposition 1.2.3, page 275).
In \cite{gi}, Givental constructs Lagrangian embeddings, into $\R^4$, of  compact nonorientable surface with Euler characteristic equal to zero modulo four.
The Gromov-Lees Theorem implies that one can find a Lagrangian immersion of any hyperbolic three-manifold into $(\R^6, \om_6)$. By taking products of such examples, it follows that there is a
split hyperbolic Lagrangian submanifold 
of $(\R^{2n}, \om_{2n})$ for $n=2$ and all $n>3$.
The lemma then follows easily from Darboux's theorem.
\end{proof}

\begin{Remark}
In contrast to the split hyperbolic case, there are more subtle obstructions to the existence of Lagrangian submanifolds 
which admit metrics of negative sectional curvature. A result of Viterbo states that there can be no such Lagrangian submanifold in any uniruled symplectic manifold of dimension greater than or equal to six, \cite{vi3, egh}. 
Hence, there are no split hyperbolic Lagrangian submanifolds of $(\R^{6}, \om_{6})$. It is not known to the authors whether  there exists a displaceable hyperbolic Lagrangian submanifold in any six-dimensional symplectic manifold. 
\end{Remark}

\medskip
We are now in a position to state the main result of the paper.

\begin{Theorem}
\label{thm}
Let $(M, \om)$ be a rational and proportional symplectic manifold of dimension $2n$ which is either closed or convex. If $L$ is 
an easily displaceable Lagrangian submanifold of $(M, \om)$ which is split hyperbolic, then $N_L \leq n+2.$ 
If, in addition,  $L$ is orientable, then $N_L \leq n+1.$
\end{Theorem}
More precisely,  we prove that if $L=P_1 \times \dots \times P_k$ is (un)orientable, then there is an element $[w] \in \pi_2(M,L)$ such that 
\begin{equation*}
\label{ }
\dim P_1 -1 \leq \masl ([w]) \leq n +1 (+1). 
\end{equation*}

One can not expect similar bounds to hold for the minimal Maslov number of general Lagrangian submanifolds. Consider, for example, the quadric $$M = \left\{[z_0, \dots z_{n+1}] \subset \C P^{n+1} \mid z_0^2 + \dots + z_n^2 = z_{n+1}^2 \right\},$$ which is both rational and proportional when equipped with the symplectic form inherited from $\C P^{n+1}$.
The real quadric $L \subset M$ is a Lagrangian sphere with minimal Maslov number $N_L= 2n$.\footnote{The authors are grateful
to Paul Biran for pointing this example out to them.}

On the other hand, for displaceable Lagrangian submanifolds,  the bounds in Theorem \ref{thm} are nearly sharp. In particular,  
in \cite{po1}, Polterovich constructs  examples of orientable monotone Lagrangian submanifolds of $(\R^{2n}, \om_{2n})$
with minimal Maslov number $k$ for every integer $k$ in $[2,n]$. 

For the special class of displaceable 
Lagrangians considered in Theorem \ref{thm}, one might expect stronger restrictions on the Maslov class similar to those obtained for $(\R^{2n}, \om_{2n})$ in Theorem \ref{vi} and by Fukaya in \cite{fu}.
Indeed, this is the case if one makes further assumptions on $(M,\om)$. In the sequel to this paper, we will study Maslov class rigidity for displaceable Lagrangians in closed symplectic manifolds which are symplectically aspherical. In that setting, we will obtain finer rigidity statements as well as rigidity statements for a larger classes of Lagrangian submanifolds which includes Lagrangian tori. 

\subsection{Organization} In the next section, we recall the definitions of and  basic results concerning convex symplectic manifolds, Hamiltonian flows, the Hofer length functional, and Sikorav's curve shortening method. 
In Section 3,  we construct and study a special Hamiltonian $H_L$ which is supported in a small neighborhood of our Lagrangian submanifold $L$. We then establish a relation between the Conley-Zehnder indices of the nonconstant contractible periodic orbits of $H_L$ and the Maslov indices of their spanning discs in Section 4.  In Section 5, we state a result, Theorem \ref{cap}, which relates the length minimizing properties of $H_L$ to its periodic orbits. Together with the previous results, this is shown to imply Theorem \ref{thm}. The proof of Theorem \ref{cap} is contained in Section 6, and some generalizations of our main result are described in Section 7.


\section{Preliminaries}

\subsection{Convex symplectic manifolds}
A compact  symplectic manifold $(M,\om)$ is said to be {\bf closed} if the boundary of $M$, $\p M$, is empty. It is called {\bf convex} if  $\p M$ is nonempty and there is a $1$-form $\alpha$ on $\p M$ such that $d \alpha = \om |_{\p M}$ and the form $\alpha \wedge (d \alpha)^{n-1}$ is a volume form on $\p M$ which induces the outward orientation. Equivalently, a compact symplectic manifold 
$(M,\om)$ is convex if there is a vector field $X$ defined in a neighborhood of $\p M$ which is transverse to $\p M$, outward pointing, and 
satisfies $\LL_X \om =\om$. For some $\eps_0 >0$, the vector field $X$ can be used to symplectically identify a neighborhood of $\p M$ with the submanifold $$M_{\eps_0}= \p M \times (-\eps_0, 0],$$ equipped with coordinates $(x,\tau)$ and the symplectic form $d(e^{\tau} \alpha)$. We will use the notation $M_{\eps}= \p M \times (-\eps, 0]$
for $0<\eps \leq \eps_0$.

A noncompact symplectic manifold $(M,\om)$  is called convex if there is an exhausting sequence 
of compact convex submanifolds $M_j$ of $M$, i.e., $M_1 \subset M_2 \subset \dots \subset M$ and  $$\bigcup_j M_j =M.$$ 

\subsection{Hamiltonian flows}
\label{hamiltonian}
A function $H \in C^{\infty}(S^1 \times M)$ will be referred to as a {\bf Hamiltonian} on $M$. 
Here, we identify the circle $S^1$ with $\R / \Z$ and parameterize it with the coordinate $t \in [0,1]$.
Set  $H_t(\cdot) = H(t, \cdot)$ and 
let $C_0^{\infty}(S^1 \times M)$ be the space of Hamiltonians $H$ such that the support of  $dH_t$ is 
compact and does not intersect  $\p M$ for all $t \in [0,1]$.
Each $H \in C^{\infty}_0(S^1 \times M)$ determines a $1$-periodic time-dependent 
Hamiltonian vector field $X_H$ via Hamilton's equation
\begin{equation*}
\label{ }
i_{X_H}\om = -dH_t.
\end{equation*}
The time-$t$ flow of $X_H$, $\phi^t_H$, is defined for all $t \in [0,1]$ (in fact, for all $t \in \R$). 
The group of Hamiltonian diffeomorphisms of $(M,\om)$ is 
the set of time-$1$ flows obtained in this manner, 
\begin{equation*}
\label{ }
\Ham(M,\om) =\{ \phi = \phi^1_H \mid H \in C_0^{\infty}(S^1 \times M) \}.
\end{equation*}

\begin{Remark}
In this work,  we are only concerned with compact Lagrangian submanifolds 
which are displaceable by Hamiltonian diffeomorphisms. By definition, Hamiltonian
diffeomorphisms are trivial away from a compact set. Hence, it suffices for us to prove Theorem \ref{thm} for closed symplectic manifolds and \emph{compact} convex symplectic manifolds. In particular, the noncompact convex case can be reduced to the 
compact convex case by restricting attention to some element $M_j$ of an exhausting sequence for $M$, for sufficiently large $j$.
\end{Remark}

In the case of a compact convex symplectic  manifold we will also consider
Hamiltonian flows which are nontrivial near $\p M$ but which are still defined for all $t \in 
\R$.
A function $f \in C^{\infty}(M)$ is said to be { \bf admissible}
if for some $\eps$ in $(0, \eps_0]$ we have 
\begin{equation}
\label{ }
f|_{\scriptscriptstyle{M_{\eps}}} (x,\tau)=a e^{-\tau} + b,
\end{equation}
where $a$ and $b$ are arbitrary constants and  $a<0$.

A Hamiltonian $H \in C^{\infty}(S^1 \times M)$
is called {\bf pre-admissible}  if for some $\eps$ in $(0, \eps_0]$ we have
\begin{equation}
\label{adm}
H|_{\scriptscriptstyle{S^1 \times M_{\eps}}}(t,x,\tau)= a e^{-\tau} +b(t),
\end{equation}
 where $a$ is again a negative constant  and $b(t)$ is a  smooth $1$-periodic function.
 We will refer to $a$ as the {\bf slope} of $H$.
%
The prescribed behavior of $H$
on $M_{\eps}$ implies that its Hamiltonian flow is defined for all $t \in \R$.
In particular, consider the Reeb vector field $R$ on $\p M$
which is defined uniquely by the following conditions; 
\begin{equation*}
\label{ }
\om(R(x),v(x))=0  \text{  and  } \om(X(x), R(x)) =1
\end{equation*}
for all $x \in \p M$ and $v \in T_x \p M.$ If $H $ has the  
form \eqref{adm} on $M_{\eps}$, then
$$
X_H(t,x,\tau) = -a R(x) \,\,\,\, \text{for all $(t,x,\tau) \in S^1 \times M_{\eps}$},
$$
and so the level sets $\{\tau =constant\}$ in $M_{\eps}$ are preserved by $\phi^t_H$.

Let $T_R$ by the minimum period of the closed orbits of $R$.
We say that a pre-admissible Hamiltonian $H$  is {\bf admissible} if its slope $a$
satisfies $-a < T_R.$
If $(M,\om)$  is closed, 
then the space of admissible functions on $M$ is simply $C^{\infty}(M)$ and the space of
admissible Hamiltonians is $C^{\infty}(S^1 \times M)$.
For simplicity, in either the closed or convex case, 
we will denote  the space of admissible Hamiltonians by $\widehat{C}^{\infty}(S^1 \times M)$. 

For any Hamiltonian flow $\phi^t_H$ defined for $t \in [0,1]$, we will denote the set of contractible $1$-periodic orbits of $H$ by $\PP(H)$. Note that if $(M, \om)$ is compact and convex and $H$ is in $C^{\infty}_0(S^1 \times M)$ or in $\widehat{C}^{\infty}(S^1 \times M)$, then the elements of $\PP(H)$ are contained in the complement of $M_{\eps}$ for some $\eps \in (0, \eps_0]$.  

An element $x(t) \in \PP(H)$ is said to be nondegenerate 
if the linearized time-$1$ flow $d\phi^1_H \colon T_{x(0)}M \to T_{x(0)}M$ does not have one as an eigenvalue.
If every element of $\PP(H)$ is nondegenerate we will call $H$ a {\bf Floer Hamiltonian}.

\subsection{The Hofer length functional and Sikorav curve-shortening}

Following \cite{ho1},  the Hofer length of a Hamiltonian path $\phi^t_H$ is defined to be
\begin{eqnarray*}
\length (\phi^t_H) &=& \|H\| \\ 
{}&=& \int_0^1 \max_M H_t \,\,dt-  \int_0^1 \min_M H_t \,\,dt\\
{} &=& \|H\|^+ + \|H\|^-.
\end{eqnarray*}
When $M$ is compact, then a Hamiltonian  $H$ in $C_0^{\infty}(S^1 \times M)$ or $\widehat{C}^{\infty}(S^1 \times M)$ is {\bf normalized} if 
$$
\int_M H_t \, \om^n = 0
$$
for every $t$ in $[0,1]$. 
If the generating Hamiltonian $H$ is normalized, then the quantities $\|H\|^+$ and $\|H\|^-$ provide different measures of the length of $\phi^t_H$ 
called the positive and negative Hofer lengths, respectively. 
 

For a path of Hamiltonian diffeomorphisms $\psi_t$ with $\psi_0 = \id$, let $[\psi_t]$ be the class of Hamiltonian paths which are homotopic to $\psi_t$ relative to its endpoints. Denote the set of normalized  Hamiltonians which lie in $C^{\infty}_0(S^1 \times M)$
and  generate the paths in $[\psi_t]$ by 
$$
C^{\infty}_0([\psi_t]) =\{ H \in C^{\infty}_0(S^1 \times M) \mid \int_M H_t \, \om^n = 0,\,[\phi^t_H] = [\psi_t]\}.
$$  
The Hofer semi-norm of $[\psi_t]$ is then defined by 
$$
\rho([\psi_t]) = \inf_{H \in C^{\infty}_0([\psi_t])} \{ \|H\| \}.
$$
The positive and negative Hofer semi-norms of $[\psi_t]$ are defined similarly as
$$
\rho^{\pm}([\psi_t]) = \inf_{H \in C^{\infty}_0([\psi_t])} \{ \|H\|^{\pm} \}.
$$
Note, that  if $$\|H\|^{(\pm)} > \rho^{(\pm)}([\phi^t_H]),$$ 
then $\phi^t_H$  fails to minimize the (positive/negative) Hofer length in its homotopy class.

The displacement energy of a subset $U \subset (M,\om)$ is the quantity
$$
e(U,M,\om) = \inf_{\psi_t} \{ \rho([\psi_t]) \mid \psi_0=\id  \text{ and }\psi_1(U) \cap \overline{U} = \emptyset \},
$$
where $\overline{U}$ denotes the closure of $U$.
The following result relates the negative Hofer semi-norm and the displacement energy. It is a direct application of Sikorav's curve shortening technique and the reader is referred  to Lemma 4.2 of \cite{ke2} for the entirely similar proof of the analogous result for the positive Hofer semi-norm.

\begin{Proposition}\label{shorten}
Let $H \in C^{\infty}_0(S^1 \times M)$ be a time-independent  normalized Hamiltonian that is constant and equal to its maximum value on the 
complement of an open set $U \subset M$. If $U$ has finite displacement energy 
and  $\|H\|^- > 2 e(U) $, then $$\|H\|^- > \rho^-([\phi^t_H])  + \12 \|H\|^+.$$ 
In other words, the Hamiltonian path $\phi^t_H$ does not minimize the negative Hofer length in its homotopy class.
\end{Proposition}

\section{A special Hamiltonian flow near $L$}

In this section we construct  a special Hamiltonian $\hl$ whose Hamiltonian 
flow is supported in a tubular neighborhood of $L$.  The nonconstant 
contactible periodic orbits of this flow  project to perturbed geodesics  
on $L$. As well, the Hamiltonian path $\phi^t_{\hl}$ fails to minimize the 
negative Hofer length in its homotopy class. 
\subsection{Perturbed geodesic flows}
\label{geodesic}
We begin by recalling some  relevant facts about geodesic flows on $L$.
For now, we assume only that $L$ is a closed manifold without boundary.
Let $g$ be a Riemannian metric on  $L$, and consider the energy functional of $g$
which is defined on the space of smooth loops $C^{\infty}(S^1, L)$,  by
$$
\EE_g(q(t)) = \int_0^1 \12  \|\dot{q}(t)\|^2   \, dt.
$$
The critical points of $\EE_g$, $\Crit (\EE_g)$, are the closed geodesics
of $g$ with period equal to one. The closed geodesics of $g$ 
with any  positive period $T>0$ correspond to the $1$-periodic orbits of the metric $\frac{1}{T} g$, and
are thus the critical points of the functional $\EE_{\frac{1}{T}g}$. 
 
The Hessian of $\EE_g$ at a critical point $q(t)$ will be denoted by $Hess(\EE_g)_q$.
As is well  known, the space on which $Hess(\EE_g)_q$ is negative definite
is  finite-dimensional. It's dimension is, by definition, the Morse index of $q$ and will be denoted here by $\mor(q)$. 
The kernel of $Hess(\EE_g)_q$ is also finite-dimensional, and is always
nontrivial unless $L$ is a point. 

A submanifold $D \subset C^{\infty}(S^1,L)$ which consists of critical points of
$\EE_g$ is said to be  Morse-Bott nondegenerate if the dimension of the kernel of 
$Hess(\EE_g)_q$ is equal to the dimension of $D$ for every $q \in D$. 
An example of such a manifold is the set of constant geodesics of any metric on $L$. This is a  Morse-Bott nondegenerate submanifold which is diffeomorphic to $L$. The energy functional $\EE_g$ is said to be Morse-Bott if all the $1$-periodic
geodesics are contained in Morse-Bott nondegenerate critical submanifolds of $\EE_g$.
Note that if $\EE_g$ is Morse-Bott, then so is $\EE_{\frac{1}{T}g}$ for any $T>0$.

\begin{Example}
\label{negative}
If $g$ is a metric with negative sectional curvature, then the nonconstant $1$-periodic geodesics occur in 
Morse-Bott nondegenerate $S^1$-families. In particular, the unparameterized geodesics of $g$ are isolated. The Morse index of every $1$-periodic geodesic is zero.
\end{Example}

\begin{Example}
\label{split}
Let $L =P_1 \times \dots \times P_k$ be a split hyperbolic manifold as defined in Section \ref{statement}. Let $g_j$ be a metric on $P_j$ with negative sectional curvature and set $$g =g_1 + \dots +g_k.$$ The nonconstant $1$-periodic geodesics of $g$ occur in 
Morse-Bott nondegenerate critical submanifolds whose dimension is no greater than $1+ \dim P_2 + \dots \dim P_k$. 
The Morse index of these closed geodesics is again zero.
\end{Example}

%


\subsubsection{Potential Perturbations}
It will be useful for us to perturb a Morse-Bott energy functional $\EE_g$ so that the critical 
points of the resulting functional are nondegenerate. We restrict ourselves to perturbations of the following classical form
$$
\EE_{g,V}(q) = \int_0^1 \left( \12  \|\dot{q}(t)\|^2 - V(t, q(t)) \right) \, dt,
$$
where the function $V \colon S^1 \times L \to \R$ is assumed to be smooth. 
The critical points of $\EE_{g,V}$ are solutions of the equation
\begin{equation}
\label{pg}
\nabla_t \dot{q} + \nabla_g V(t,q) = 0
\end{equation}
where $\nabla_t$ denotes covariant differentiation in the $\dot{q}$-direction with respect to the Levi-Civita 
connection of $g$, and  $\nabla_g V$ is the gradient vector field of $V$ with respect to $g$. We refer to solutions
of \eqref{pg} as perturbed geodesics.

\begin{Theorem}\cite{we}
\label{potential trans}
There is dense set $\VV_{reg}(g) \subset C^{\infty}(S^1\times L)$ such that for $V \in \VV_{reg}(g)$ the 
critical points of $\EE_{g,V}$ are nondegenerate.
\end{Theorem}

When a Morse-Bott functional $\EE_g$ is perturbed, the critical submanifolds break apart into 
critical points. For a small perturbation $V$ it is possible to relate each critical point of $\EE_{g,V}$ to a specific 
critical submanifold of $\EE_g$, and to relate their indices. Here is the precise statement.

\begin{Lemma}
\label{potential}
Let $\EE_g$ be Morse-Bott and let $\Crit^a(\EE_g)= \amalg_{j=1}^{\ell} D_j $ be the 
(finite, disjoint) union of all critical submanifolds of $\EE_g$
with energy less than $a$. Let $\epsilon >0$ be small enough so that the $\epsilon$-neighborhoods 
of the $D_j$ in $C^{\infty}(S^1,L)$ are disjoint. If $V \in \VV_{reg}(g)$ is sufficiently small, then 
 each nondegenerate critical point $q(t)$ of $\EE_{g,V}$, with action less than $a$, lies in the $\epsilon$-neighborhood
 of exactly one component $D_j$ of $\Crit^a(\EE_g)$. Moreover,
\begin{equation*}
\label{ }
\mor(q) \in [\mor(D_j), \mor(D_j)+\dim(D_j)]. 
\end{equation*}
\end{Lemma}

\begin{proof}
For a proper Morse-Bott function on a finite-dimensional manifold the proof of the lemma is elementary
and requires only the generalized Morse Lemma which yields a normal form for a function near a 
Morse-Bott nondegenerate critical submanifold.
As we now describe, the present case is essentially identical to the finite-dimensional one. 
Let $\widehat{\EE}_g$ be the extension of $\EE_g$ to the space $\Lambda^1(S^1,M)$ of loops in 
$L$ of Sobolev class $W^{1,2}$. This space of loops is a Hilbert  manifold  and
the critical point sets $\Crit(\EE_g)$ and $\Crit(\widehat{\EE}_g)$ coincide, as do the corresponding
Morse indices.  Let $q$ be a critical point of $\widehat{\EE}_g$ which belongs to a critical submanifold $D$. 
By the Morse-Bott assumption, the nullity of $Hess(\widehat{\EE}_g)_q$, 
is equal to the dimension of $D$. Moreover, the only possible accumulation point for the eigenvalues of $Hess(\widehat{\EE}_g)_q$
is one, \cite{kli} Theorem 2.4.2. Hence, there are only finitely many  negative eigenvalues and the 
positive eigenvalues are uniformly bounded away from zero. Using the generalized Morse Lemma for Hilbert manifolds from \cite{me}, see also \cite{kli} Lemma 2.4.7., the result then follows exactly as in the finite-dimensional case.  
\end{proof}

\subsection{Hamiltonian geodesic flows}
Consider the cotangent bundle of $L$, $T^*L$, equipped with the 
symplectic structure $d \theta$ where $\theta$ is the 
canonical Liouville $1$-form. 
We will denote points in $T^*L$ by $(q,p)$ where $q$ is in
$M$ and $p$ belongs to $T^*_q L$.  In these local coordinates, 
$\theta = p dq$ and so $d\theta = dp \wedge dq.$

The metric $g$ on $L$  induces a bundle isomorphism between 
$TL$ and the cotangent bundle $T^*L$ and hence a cometric on $T^*L$. 
Let $K_g \colon T^*L \to \R$ be the function $K_g(q,p) = \12 \|p\|^2$. 
The Legendre transform yields a bijection between the critical points of the perturbed energy functional $\EE_{g,V}$ 
and the critical points of the  action functional
$$\AC_{{K_g}+V}\colon C^{\infty}(S^1, T^*L) \to \R$$ defined by
\begin{equation*}
\label{ }
\AC_{{K_g}+V}(x) =  \int^1_0 ({K_g}+V)(t,x(t)) \, dt - \int_{S^1} x^* \theta.
\end{equation*}
The critical points of $\AC_{{K_g}+V}$ are the $1$-periodic orbits of 
the Hamiltonian $K_g+V$ on $T^*L$. 
If $x(t)=(q(t), p(t))$ belongs to $\PP(K_g+V)$, then its
projection to $L$  is a closed $1$-periodic  solution of \eqref{pg}
with initial velocity $\dot{q}(0)$ determined by $g(\dot{q}(0), \cdot) = p(0)$.
Moreover, $x(t)$ is nondegenerate if and only if $q(t)$ is nondegenerate. 

\subsection{A perturbed geodesic flow supported near $L$}

We now assume that $L$ is a Lagrangian submanifold of $(M, \om)$ as in the 
statement of Theorem \ref{thm}. In particular, $L$ is easily displaceable and split hyperbolic. We equip $L$ with the metric $g$ from Example \ref{split} and remind the reader that the energy functional for this metric is Morse-Bott.

Consider a  neighborhood of the zero section in $T^*L$ of the following type 
$$
U_r=  \{ (q,p) \in T^*L  \mid \|p\| < r\}.
$$
By Weinstein's neighborhood theorem, for sufficiently small $r>0$, there is a neighborhood 
of $L$ in $(M, \om)$ which is symplectomorphic to $U_r$. We only consider values of $r$ 
for which this holds, and will henceforth  identify $U_r$ with a 
neighborhood of $L$ in $(M, \om)$. 
For a subinterval $I \subset [0,r)$, we will use the notation
$$
U_I = \{ (q,p) \in U_r \mid \|p\| \in I \}.
$$

\begin{Proposition}
\label{function}
For sufficiently small  $r>0$, there is a normalized and admissible Hamiltonian  
$\hl \colon S^1 \times M \to \R$ with the following properties:
  \newcounter{Lcount}
  \begin{list}{{\bf (H\arabic{Lcount})}}
    {\usecounter{Lcount}
    \setlength{\rightmargin}{\leftmargin}}
 \item The constant $1$-periodic orbits $\hl$ correspond to the critical points of 
  an admissible  Morse function $F$ on $M$. Near these points the Hamiltonian flows of $\hl$
  and $c_0 F$ are identical for some arbitrarily small constant $c_0>0$;
  \item The nonconstant $1$-periodic orbits of $\hl$ are nondegenerate and contained in $U_{(r/5+\delta, 2r/5-\delta)} \cup U_{(3r/5+\delta, 4r/5-\delta)}$  for some $0< \delta < r/5$. Each such orbit  projects to 
a nondegenerate closed perturbed geodesic $q(t)$. Moreover, if $T$ is the period of $q$, then $q$ can be associated to 
exactly one critical submanifold $D$ of $\EE_{\frac{1}{T}g}$ and 
\begin{equation*}
\label{ }
\mor(q) \in [\mor(D), \mor(D)+\dim(D)]; 
\end{equation*}
  \item There is a point $Q \in L \subset M$ which is the unique local minimum
  of $\hl(t, \cdot)$ for all $t \in [0,1]$. Moreover,
 \begin{equation}
\label{negative bound}
 \|\hl\|^- = - \int_0^1 \hl(t,Q)\,dt > 2e(U_r);
\end{equation} 
  \item The Hofer norm of $\hl$ satisfies
\begin{equation}
\label{total bound}
2e(U_r) <\|\hl\| < r(M,\om).
\end{equation}
\end{list}
\end{Proposition}

\begin{proof}
The construction of $\hl$ is elementary but somewhat involved. We divide the 
process into four steps. Some of the constituents of $\hl$ 
which are defined in these steps will be referred to in later sections.

\medskip
\noindent {\bf Step 1: A geodesic flow supported near $L$.}
\medskip

Let $\nu = \nu(A,B,C, r)  \colon [0,+\infty) \to \R$ be a smooth function with the following properties
\begin{itemize}
  \item $\nu  = -A$ on $[0, r/5]$;
  \item $\nu', \nu'' > 0$ on $(r/5, 2r/5)$; 
  \item $\nu ' = C $ on $[2r/5, 3r/5]$;
   \item $\nu' > 0$ and $\nu''<0 $ on $(3r/5, 4r/5)$; 
  \item $\nu = B$ on $[4r/5 ,+\infty)$.
\end{itemize}
Define the function $K_{\nu}$ on $T^*L$ by
$$K_{\nu}(q,p)=
\begin{cases}
 \nu(\|p\|) & \text{if $(q,p)$ is in $U_r$}, \\
  B    & \text{otherwise}.
\end{cases}
$$

Since $L$ is easily displaceable, for sufficiently small $r>0$ we have 
\begin{equation*}
\label{ }
e(U_r) < r(M,\om)/2.
\end{equation*}
Choose the constant $A$, so that 
\begin{equation}
\label{short}
2e(U_r)<A< r(M,\om).
\end{equation}
Restricting $r$ again, if necessary, we can choose 
the constant $B$ so that 
\begin{equation*}
0 < B <  A \frac{\Vol(U_r)}{\Vol(M \ssminus U_r)}.
\end{equation*}
For these choices of $r$, $A$ and $B$, we can construct $\nu$, as above,  so that 
$K_{\nu}$ is normalized and satisfies
\begin{equation}
\label{shorter}
2e(U_r)<\|K_{\nu}\|< r(M,\om).
\end{equation}
We may also choose the positive constant $C$ so that it is not the length 
of any closed geodesic of $g$.

The Hamiltonian flow of $K_{\nu}$ is trivial in both $U_{r/5}$ and the complement of $U_{4r/5}$. 
Hence, each nonconstant $1$-periodic orbit $x(t)=(q(t), p(t))$ of $K_{\nu}$ is contained in 
$U_{(r/5, 4r/5)}$ where  
\begin{equation*}
\label{ }
X_{K_{\nu}} (q,p) = \left( \frac{\nu'(\|p\|)}{\|p\|}  \right)  X_{K_g}(q,p).
\end{equation*}
Our choice of $C$ implies that all nonconstant orbits of $K_{\nu}$ occur on the level sets  
contained in $U_{(r/5 , 2r/5)}$ or  $U_{(3r/5 , 4r/5 )}$,
where $\nu$ is convex or concave, respectively.
In fact, these nonconstant 
orbits lie in $$U_{(r/5 +\delta, 2r/5 -\delta)} \cup U_{(3r/5 + \delta, 4r/5 - \delta)}$$ for some $r/5>\delta>0$. This follows from the fact that
$dK_{\nu}$ equals zero
along the boundary of $U_{(r/5, 2r/5)} \cup U_{(3r/5, 4r/5)}$.

\medskip
\noindent {\bf Step 2: A Morse function isolating $L$.}
\medskip

Let $F_0 \colon M \to \R$ be an admissible Morse-Bott function 
with the following properties:
\begin{itemize}
\item The submanifold $L$ is a critical submanifold with index equal to $0$.
 \item On $U_r$,  we have $F_0=f_0(\|p\|)$ for some increasing function $f_0 \colon [0,r] \to \R$ which 
is strictly convex on $[0, 2r/5)$, linear on $[2r/5, 3r/5]$, and strictly concave on $(3r/5,r]$.
\item All critical submanifolds other than $L$ are isolated nondegenerate critical points with strictly positive Morse indices.
\end{itemize}
Such a function is easily constructed by starting with the square of the distance function from $L$ with respect to a metric which coincides, in the normal directions, with the cometric of $g$ inside $U_r$. 
This distance function can then be deformed within $U_r$ to obtain the first and second properties above. Perturbing the resulting function away from $U_r$, one can ensure that it is Morse there. Then performing some elementary handle slides and cancelations away from $U_r$, as in \cite{mi}, one can get rid of all other local minima.
 
Let $F_L \colon L \to \R$ be a Morse function with a unique local minimum at a point $Q$ in $L$.
Choose a bump function $\hat{\sigma} \colon [0,+\infty) \to \R$ such that $\hat{\sigma}(s)=-1$ for $s$ near zero and 
$\hat{\sigma}(s)=0$ for $s \geq r/5$. Let $\sigma = \hat{\sigma}(\|p\|)$ be the corresponding function on $M$ with 
support in $U_{r/5}$ and set
$$
F= F_0 + \eps_L \cdot \sigma \cdot F_L.
$$
For a sufficiently small choice of $\eps_L>0$, $F$ is a Morse 
function whose critical points away from $U_{r/5}$ agree with those of 
$F_0$ and whose critical points in $U_{r/5}$ are precisely the critical points of 
$F_L$ on $L \subset M$  (see, for example, \cite{bh} page 87).

\medskip
\noindent {\bf Step 3: Pre-perturbation.}
\medskip

For $c_0>0$, consider the function
$$
H_0 = K_\nu + c_0 F.
$$
By construction, we have
$$
H_0 =
\begin{cases}
     -A+c_0 F & \text{on $U_{r/5}$}, \\
     (\nu + c_0 f_0)(\|p\|) & \text{on $U_{[r/5, 4r/5]}$},\\
     B + c_0 F & \text{elsewhere}.
\end{cases}
$$
From this expression it is clear that each $H_0$ is a Morse function with 
$\Crit(H_0) =\Crit(F)$. As well, $Q$ is the unique local minimum of $H_0$. 
Moreover, when $c_0$ is sufficiently small, the nonconstant $1$-periodic 
orbits of $H_0$, like those of $K_{\nu}$,  are contained 
in 
$$U_{(r/5 +\delta, 2r/5-\delta)} \cup U_{(3r/5+ \delta, 4r/5- \delta)}$$
for some $r/5> \delta>0$. 
%


\medskip
\noindent {\bf Step 4: Perturbation of $H_0$.}
\medskip

The function $H_0$ has properties {\bf (H1)}, {\bf(H3)} and {\bf(H4)}. To obtain a function with property {\bf(H3)}
we must  perturb $H_0$ so that the $1$-periodic orbits of the resulting Hamiltonian are nondegenerate.

Let $N^{\rho}$ be a critical submanifold of $\AC_{H_0}$ which is contained in the level set $\|p\| = \rho$. 
Denote the  projection of $N^{\rho}$ to $L$ by $D^{\rho}$. Then $D^{\rho}$ is a Morse-Bott nondegenerate set of 
periodic geodesics with period $\left((\nu' + c_0f_0')({\rho})\right)^{-1}$. Alternatively, 
$D^{\rho}$ can be viewed as a collection of $1$-periodic  geodesics of the 
metric $(\nu' + c_0f_0')({\rho})g$. We will adopt this  latter point of view. 

There are finitely many critical submanifolds of $\AC_{H_0}$.
We label the submanifolds of $1$-periodic closed geodesics
 which appear as their projections by
$$
\{ D^{{\rho}_j}_j  \mid j=1, \dots \ell \}.
$$
Theorem \ref{potential trans}, implies that the set of potentials $$\bigcap_{j=1, \dots, \ell} \VV_{reg}((\nu' + c_0f_0')({\rho}_j) g)$$
is dense in $C^{\infty}(S^1 \times L)$.  We can choose a $V$ in this set
which is arbitrarily small  with respect to the $C^{\infty}$-metric. By 
Lemma \ref{potential}, the projection, $q(t)$,  of  each
$1$-periodic orbit $x(t)$ of $\nu({K_g} + V)$ is then nondegenerate and 
lies arbitrarily close to (within a fixed distance of ) exactly one of the  $D^{{\rho}_j}_j$ and  satisfies
\begin{equation*}
\label{ }
\mor(q) \in [\mor(D^{{\rho}_j}_j), \mor(D^{{\rho}_j}_j)+\dim(D^{{\rho}_j}_j)]. 
\end{equation*}

Given such a $V$ we define $V_0 \colon S^1 \times M \to \R$  so that
$$
V_0= 
\begin{cases}
 V(t,q)     & \text{, in $U_{(2r/5 +\delta, 3r/5-\delta)}$} \\
  0   & \text{, on the complement of $U_{(2r/5, 3r/5)}$}.
\end{cases}
$$ 
Clearly, the function $V_0$,  like $V$ itself, can be chosen to be arbitrarily small. 
We then define $\hl$ by 
$$
\hl(t,q,p)=
\begin{cases}
   (\nu+c_0 f_0)\left(\sqrt{\12 \|p\|^2 +V_0(t,q,p)}\right)   & \text{for $(q,p)$ in $U_{(2r/5, 3r/5)}$}, \\
    H_0(q,p)  & \text{otherwise}.
\end{cases}
$$
This function has properties {\bf(H1)} - {\bf(H4)} as desired. 
It only remains to normalize $\hl$. Adding the function  
$$-\int^1_0 \hl(t, \cdot) \, \om^n$$ 
to $\hl$ we get a  normalized Hamiltonian. Since $\hl$ is a small perturbation
of $K_{\nu}$, this new Hamiltonian still has properties {\bf (H1)} - {\bf (H4)}. For simplicity, we still denote it by $\hl$.
%

\end{proof}

\subsection{Additional properties of $\hl$}
Let $H$ be Floer Hamiltonian. A spanning disc for a $1$-periodic orbit $x$ in $\PP(H)$
is a smooth map $ w \colon D^2 \subset \C \to M$ such that $w(e^{2\pi i t}) = x(t)$. The action of $x$ with respect to $w$ is defined as 
\begin{equation*}
\label{ }
\AC_{H}(x,w) = \int_0^1H(t,x(t)) \, dt - \int_{D^2} w^* \om.
\end{equation*}
A spanning disc also determines a homotopy class of trivializations of $x^*TM$
and hence a Conley-Zehnder index $\cz(x,w)$. This index is described in \S \ref{cz}. We only mention here that it is normalized so that for a critical point $p$  of a $C^2$-small Morse function, 
and the constant spanning disc $w_p(z)=p$, we have  
\begin{equation*}
\label{ }
\cz(p,w_p)= \12 \dim M - \mor(p).
\end{equation*}
 
The following result provides a simple criterion for recognizing nonconstant periodic orbits of  the Hamiltonian $\hl$ on $(M^{2n}, \om)$ constructed in Proposition \ref{function}. 
 
\begin{Lemma}\label{nonconstant}
If $x$ is a contractible $1$-periodic orbit of $\hl$ which admits a spanning disc $w$
such that $-\|\hl\|^- < \AC_{\hl}(x,w) \leq \|\hl\|^+$ and  $\cz(x,w)=n$, then $x$ is nonconstant.
\end{Lemma}

\begin{proof}
Arguing by contradiction, we assume that $x(t)=P$ for some point $P$ in $M$. 
The spanning disc $w$ then represents an element 
$[w]$ in $\pi_2(M)$, and we have 
\begin{equation*}\label{action}
\AC_{\hl}(x,w) = \int_0^1 \hl(t,P) \,dt - \om([w]).
\end{equation*}
By property {\bf (H1)}, the point $P$ corresponds to a critical point of $F$ and the Hamiltonian 
flow of $\hl$ and $c_0 F$ are identical near $P$. The normalization of the 
Conley-Zehnder index  then yields
\begin{equation}\label{indexi}
\cz(x,w) = n - \mor(P) + 2c_1([w]).
\end{equation}

If $\om([w])=0$, then the proportionality of $(M,\om)$ implies that $c_1([w]) = 0$.\footnote{This is the only point where we use the assumption that $(M, \om)$ is proportional.} It then follows from \eqref{indexi} that 
the Morse index of $P$ must be zero. 
This implies that $P=Q$, since $Q$ is the unique fixed local minimum of $\hl$. However, the action of $Q$ 
with respect to a spanning disc $w$ with $\om([w])=0$ is equal to 
$-\|\hl\|^-$. This is outside the assumed action range and hence a contradiction.

We must therefore have $\om([w])\neq 0$ and thus, by the definition of $r(M, \om)$ and property {\bf (H4)}
$$|\om([w])| \geq r(M,\om)> \|\hl\|.$$
If $\om([w])>0$, then
\begin{eqnarray*}
\AC_{\hl}(x,w) & \leq  & \int_0^1 \hl(t,P)\,dt  - \|\hl\| \\
{} & = & \int_0^1 \hl(t,P)\,dt - \|\hl\|^+  - \|\hl\|^-\\
{} & \leq &  - \|\hl\|^-
\end{eqnarray*}
which is again outside the assumed action range.

For the remaining case, $\om([w])<0$, we have  
$$
\AC_{\hl}(x,w) \geq \int_0^1 \hl(t,P) \,dt + \|\hl\| = \int_0^1 \hl(t,P)\,dt + \|\hl\|^+ + \|\hl\|^-.
$$
Hence, either $\AC_{\hl}(P,w) > \|\hl\|^+$ or $P=Q$. Both of these conclusions again contradict our hypotheses, and so $x$ must be nonconstant. 

\end{proof}

Periodic orbits meeting the previous criteria will be detected using the techniques  developed in 
\cite{ke1} to study length minimizing Hamiltonian paths. The following intermediate result will be crucial 
in applying these methods.
\begin{Lemma}
\label{short-g}
There is a normalized admissible Hamiltonian $G_{\scriptscriptstyle{L}}$ on $M$ such that 
\begin{enumerate}
    \item the admissible Hamiltonian  path
$\phi^t_{G_{\scriptscriptstyle{L}}}$ is homotopic to $\phi^t_{\hl}$ relative its end points;
  \item $\|G_{\scriptscriptstyle{L}}\|^- < \|H_L\|^-$.
\end{enumerate}
\end{Lemma}

\begin{proof}
First we note that inequality \eqref{shorter} and Proposition \ref{shorten}, together imply that the path $\phi^t_{K_{\nu}}$ does not minimize the 
negative Hofer length in its homotopy class. Hence, 
there is a normalized Hamiltonian $G_{\nu}$ in $C^{\infty}_0([\phi^t_{K_{\nu}}])$ such that 
\begin{equation}\label{g2h}
\|K_{\nu}\|^- \geq \|G_{\nu}\|^- + 2\zeta
\end{equation}
for some $\zeta>0.$ 


Now consider the Hamiltonian flow $\phi^t_{\hl} \circ (\phi^t_{K_{\nu}})^{-1}$  which is generated 
by the Hamiltonian 
$$
F_{\nu} =  \hl - K_{\nu}\circ \phi^t_{K_{\nu}} \circ (\phi^t_{\hl})^{-1}.
$$
By construction,  $\hl$ is arbitrarily close to $K_{\nu}$ in the $C^{\infty}$-topology. Hence,
$F_{\nu}$ is also arbitrarily close to zero in the $C^{\infty}$-topology. 

The Hamiltonian flow 
\begin{equation}
\label{short}
\phi^t_{\hl} \circ (\phi^t_{K_{\nu}})^{-1} \circ \phi^t_{G_{\nu}}
\end{equation}
is generated by the normalized Hamitlonian
$$
G_{\scriptscriptstyle{L}} = F_{\nu} +G_{\nu} \circ (\phi^t_{F})^{-1}.
$$
Since $F_{\nu}=\hl - B$ near the boundary of $M$ and $G_{\nu}$ belongs to $C^{\infty}_0(S^1 \times M)$, we see that
$\gl$ is also admissible. 
The flow of $G_L$ is  homotopic to $\phi^t_{\hl}$ via the homotopy of admissible Hamiltonian flows
$$
s \mapsto \phi^t_{\hl} \circ (\phi^t_{sK_{\nu}})^{-1} \circ \phi^t_{sG_{\nu}}.
$$
Finally, by choosing $\| \hl - K_{\nu} \|_{C^{\infty}}$ to be sufficiently small, we  have 
\begin{eqnarray*}
\|G_{\scriptscriptstyle{L}}\|^-  & \leq &   \|F_{\nu}\|^- + \|G_{\nu}\|^- \\
{} & \leq & \|F_{\nu}\|^- + \|K_{\nu}\|^- - 2 \zeta \\
{} & \leq & \|\hl\|^- +  \|\hl - K_{\nu}\|^- +\|F_{\nu}\|^- - 2\zeta \\
{} & \leq & \|\hl\|^- - \zeta.
\end{eqnarray*}

\end{proof}

\section{Index relations}

Let $x$ be a $1$-periodic orbit of $\hl$ with spanning disc $w$.
In this section we establish an identity, \eqref{basic}, which relates the Conley-Zehnder index 
of $x$ with respect to $w$, the Maslov index of $w$, and the Morse index of the 
perturbed geodesic determined by $x$.  For a split hyperbolic Lagrangian submanifold $L$,  this identity yields crucial bounds for the Maslov index of $w$ which depend on the dimensions of the factors of $L$ and the  Conley-Zehnder index of $x$ with respect to $w$. The results presented in this section are not new, but we were unable to find a reference for all of them which was suitable for our purposes. 

\subsection{Basic indices}
\label{basic indices}

There are two classical versions of the Maslov index in symplectic linear algebra (see, for example, \cite{mcsa:book}). 
The first of these indices is defined for continuous loops in
$\Lambda_{2n}$, the set of Lagrangian subspaces of $(\R^{2n}, \om_{2n})$. 
We denote  this index by $\mas$. As noted by Arnold in \cite{ar}, it can be
defined as an intersection number. We now recall the generalization of this interpretation from \cite{rs}.

Let $\eta \colon S^1 \to \Lambda_{2n}$  be a loop of Lagrangian subspaces and  let $V \in \Lambda_{2n}$ be a fixed reference space. One calls $t_0 \in S^1$ a {\bf crossing}  
of $\eta$ (with respect to $V$) if $\eta(t_0)$ and $V$ intersect nontrivially. At a crossing $t_0$,  one can define a crossing form $Q(t_0)$ on 
$\eta(t_0)\cap V$ as follows. Let $W \in \Lambda_{2n}$ be transverse to $\eta(t_0)$. For each $v$ in $\eta(t_0)\cap V$  we define, for $t$ near $t_0$, the 
path $w(t)$ in $W$ by $$v+ w(t) \in \eta(t).$$ We then set $$Q(t_0)(v) = \frac{d}{dt}\Big|_{t=t_0} \om (v, w(t)).$$ The crossing $t_0$
is said to be {\bf regular} if $Q(t_0)$ is nondegenerate. If all the crossings of $\eta$ are regular then they are isolated and the Maslov index is defined 
by 
\begin{equation}
\label{maslov}
\mas(\eta;V) = \sum_{t \in S^1} sign(Q(t)),
\end{equation}
where $\emph{sign}$ denotes the signature and the sum is over all crossings. This integer is independent of the choice of $V$ (as well as the choices of $W$ at each crossing).

The second classical Maslov index is defined for continuous loops $\gamma \colon S^1 \to Sp(2n)$ where $Sp(2n)$ is the group of $2n \times 2n$ real matrices which preserve $\om_{2n}$. We denote this integer by $m(\gamma)$
and refer the reader to \cite{mcsa:book} for its definition.
We note that it is related to the Maslov index for loops of Lagrangian subspaces in the following way. 
Recall that the graph of  a matrix $A \in Sp(2n)$, $\Gamma_A$, is a Lagarangian subspace 
of the product $\R^{4n}=\R^{2n} \times \R^{2n}$  equipped with the 
symplectic form $\om_{2n}  \times (-\om_{2n})$. To a loop $\gamma(t)$ in $Sp(2n)$ 
one can then associate a loop of Lagrangian subspaces  $\Gamma_{\gamma(t)}$ in $\R^{4n}$. In this case,  
\begin{equation}
\label{index ident}
2m(\gamma) = \mas(\Gamma_{\gamma}).
\end{equation}

One can also define a Maslov-type index for 
certain paths in $Sp(2n)$. This was first defined by Conley and Zehnder in \cite{cz}.
Let $Sp^*(2n)$ be the subset of $Sp(2n)$ which consists of matrices which do not have one as an eigenvalue.
Set
$$
\mathbf{Sp}(2n) = \left\{ \Phi \in C^0 ([0,1], Sp(2n)) \mid \Phi(0) = \id,\, \Phi(1) \in Sp^*(2n) \right\}.
$$
The Conley-Zehnder index associates an integer, $\cz(\Phi)$,  to any path $\Phi$ in $\mathbf{Sp}(2n)$. 
It can be defined axiomatically,  and the relevant axioms for the present work are:
 \begin{description}
  \item[Homotopy Invariance] The index $\cz$ is constant on the components of $\mathbf{Sp}(2n)$; 
   \item[Loop Property] If $\gamma \in \mathbf{Sp}(2n)$ satisfies $\gamma(0) = \gamma(1)$, then
  \begin{equation*}
\label{ }
\cz(\gamma \circ \Phi) = 2m(\gamma) + \cz(\Phi)
\end{equation*} 
for every $\Phi \in \mathbf{Sp}(2n)$;
  \item[Normalization] The Conley-Zehnder index of the path $e^{ t \pi i}$ in $\mathbf{Sp}(2)$ is one.\footnote{This normalization differs by a minus sign from the one used in \cite{ke1}.}
\end{description}

\subsection{Nonlinear versions of the basic indices}  

The basic indices described above can be used to define
useful indices in a variety of nonlinear settings. We recall three 
such examples which will be used in the proof of Theorem \ref{thm}.

\subsubsection{The Maslov class of a Lagrangian submanifold}
\label{maslov}

We begin with the definition of  the Maslov class, $\masl \colon \pi_2(M,L) \to \Z,$ 
of a Lagrangian submanifold $L$ of $(M,\om)$. Any continuous representative $w \colon (D^2, \p D^2) \to (M,L)$ of a class in $\pi_2(M,L)$
determines a symplectic trivialization of $q^*(TM) = q^*(TT^*L)$, where
$q(t) = w(e^{2\pi i t}).$ 
Let
$$
\Phi_w \colon S^1 \times \R^{2n} \to q^*(TT^*L).
$$
be such a trivialization. Recall that  the verical subbundle $Vert$ of $T(T^*L)$
is a Lagrangian subbundle. The trivialization $\Phi_w$ then yields a loop 
$\eta_w(t) = \Phi_w(t)^{-1}(Vert(q(t))$ of Lagrangian subspaces of $\R^{2n}$.
One then defined
\begin{equation*}
\label{ }
\masl([w]) = \mas(\eta_w).
\end{equation*}



\subsubsection{Contractible periodic orbits of Hamiltonian flows}
\label{cz}

One can also define a Conley-Zehnder  index for the contractible 
nondegenerate periodic orbits of a general Hamiltonian flow.
Let $H$ be a Hamiltonian on $(M, \om)$ and let $x\colon S^1 \to M$ be a contractible and nondegenerate $1$-periodic orbit of $X_H$. A spanning disc for $x$, $w \colon D^2 \to M$, 
 determines a symplectic trivialization
\begin{equation*}
\label{ }
\Phi_w \colon S^1 \times \R^{2n} \to x^*(TM).
\end{equation*}
The Conley-Zehnder index of $x$ with respect to $w$ is then defined by 
\begin{equation*}
\label{ }
\cz(x,w) = \cz \left( \Phi_w(t)^{-1} \circ (d \phi^t_H)_{x(0)} \circ \Phi_w(0)\right)
\end{equation*}

By the homotopy invariance property of the Conley-Zehnder index, $\cz(x,w
)$ depends only on the homotopy class of the spanning disc $w$, relative its boundary.
Changing this homotopy class by gluing a representative of the class $A \in \pi_2(M)$
to the map $w$, in the obvious way,  has the following effect
\begin{equation*}
\label{ }
\cz(x,A\#u) = \cz(x,u) + 2 c_1(A).
\end{equation*}


The normalization of $\cz$ implies that if $p$ is a critical point of a $C^2$-small Morse function $H$, 
and $w_p(D^2) =p$ is the constant spanning disc, then 
\begin{equation*}
\label{ }
\cz(p,w_p)= \12 \dim M - \mor(p).
\end{equation*}
This fact was used in the proof of Lemma \ref{nonconstant}.

\subsubsection{Closed perturbed geodesics.}

Finally, we recall the definition of a Conley-Zehnder index associated to
closed orbits of perturbed geodesic flows.
Consider a Hamiltonian $H \colon S^1 \times T^*L \to \R$
of the form $$H(t,q,p)= \12 \|p\|^2 + V(t, q),$$
and let
$
x(t)=(q(t),p(t))
$
be a nondegenerate $1$-periodic orbit of $H$. 
Recall from Section \ref{geodesic} that
the projection $q(t)$ is a 
perturbed geodesic, i.e., a critical point of the energy functional
$$
\EE_g(q) = \int_0^1 \left( \12  \|\dot{q}(t)\|^2 + V(t, q(t)) \right) \, dt.
$$
As such, $q$ has a finite Morse index, $\mor(q)$ which 
 is the number of negative eigenvalues of the Hessian of $\EE_g$ at $q$, 
counted with multiplicity.

To define an index of  Conley-Zehnder type for $x$, one does not use  a symplectic  trivialization of $x^*(T^*L)$ determined by a spanning disc. Instead, as we describe below,
one uses an intrinsic class of trivializations which are determined by a global Lagrangian splitting of $TT^*L$.  
This yields a Conley-Zehnder index for both contractible and noncontractible orbits in $T^*L$ which we will 
refer to as the {\bf internal Conley-Zehnder}  and will denote by $\czi$. 

For any point $x=(q,p) \in T^*L$, the Levi-Civita connection for the metric $g$ determines a splitting
$$T_x(T^*L)= Hor(x) \oplus Vert(x)$$
into horizontal and vertical subbundles. The vertical bundle $Vert(x)$ is a Lagrangian subbundle and is 
canonically isomorphic to $T^*_q L$. The horizontal bundle $Hor(x)$ is canonically isomorphic to $T_qL$ and is 
also Lagrangian, \cite{kli}. Hence we can identify $x^*(T T^*L)$ with $q^*(TL) \oplus q^*(T^*L).$ Note that while the symplectic vector bundle $x^*(TT^*L)$ is always trivial, the factors $q^*(TL)$ and  $q^*(T^*L)$ need not be. 

If $q^*(TL)$ is trivial, for example if L is orientable, then a trivialization $\psi \colon S^1 \times \R^n \to  q^*(TL)$
determines a trivialization $\Psi$ of $T_{x(t)} T^*L = T_{q(t)}L \oplus T^*_{q(t)}L$ as follows,
\begin{eqnarray*}
\Psi \colon S^1 \times  \R^n \times \R^n  & \to & T_{q(t)}L \oplus T^*_{q(t)}L \\
(t,v,w) & \mapsto & 
\begin{pmatrix}
  \psi(t)    &    0\\
     0 &  (\psi(t)^*)^{-1}
\end{pmatrix}
\begin{pmatrix}
      v    \\
      w  
\end{pmatrix} 
\end{eqnarray*}
The {\bf internal Conley-Zehnder index} of $x$ is then defined by
\begin{equation*}
\label{ }
\czi(x) = \cz(\Psi(t)^{-1} \circ (d \phi^t_H)_{x(0)} \circ \Psi(0))
\end{equation*} 
This index  does not depend on the choice of  the 
 trivialization $\psi$ (see Lemma 1.3 of \cite{as}).

When $q^*(TL)$ is nontrivial, we proceed as in \cite{we}. Consider a map
$\psi \colon [0,1] \times \R^n \to  q^*(TL)$ such that 
\begin{equation}
\label{match}
\psi(1) = \psi(0) \circ E_1,
\end{equation}
where $E_1 \colon \R^n \to \R^n$ is the diagonal $n\times n$ matrix with diagonal 
$(-1,1, \dots, 1)$. Equip $\R^n \times \R^n$ with coordinates $(x_1, \dots, x_n,y_1,\dots, y_n)$, and extend $\psi$ to a symplectic trivialization
\begin{eqnarray*}
\Psi \colon S^1 \times  \R^n \times \R^n  & \to & T_{q(t)}L \oplus T^*_{q(t)}L \\
(t,v,w) & \mapsto & 
\begin{pmatrix}
  \psi(t)    &    0\\
     0 &  (\psi(t)^*)^{-1}
\end{pmatrix}
U(t)
\begin{pmatrix}
      v    \\
      w  
\end{pmatrix}
\end{eqnarray*}
where $U(t)$ is the rotation of the $x_1y_1$-plane in $\R^n \times \R^n$ by 
$-t\pi$ radians. Again, we set 
\begin{equation*}
\label{ }
\czi(x) = \cz(\Psi(t)^{-1} \circ (d \phi^t_H)_{x(0)} \circ \Psi(0)),
\end{equation*} 
and note that this index is also independent of the choice of the initial
 trivialization $\psi$ satisfying \eqref{match}.

\subsection{Relations between indices}

The first relation we discuss is between the internal Conley-Zehnder index of
a nondegenerate $1$-periodic orbit $
x(t)=(q(t),p(t))
$ of the Hamiltonian $H(t,q,p)= \12 \|p\|^2 + V(t, q)$ on $(T^*L, d\theta)$,
and the Morse index of the closed perturbed geodesic $q(t)$. 
The following result was proven by Duistermaat in \cite{du}. An alternative proof, as well as the extension to the nonorientable case, is contained in \cite{we}.
\begin{Theorem} \cite{du,we}
Let $x(t)=(q(t),p(t))$ be a nondegenerate $1$-periodic orbit of $H(t,q,p)= \12 \|p\|^2 + V(t, q).$ 
If $q^*(TL)$ is trivial then 
\begin{equation}
\label{orient}
\czi(x) = \mor(q).
\end{equation}
Otherwise, 
\begin{equation}
\label{nonorient}
\czi(x) = \mor(q)-1.
\end{equation}
\end{Theorem}

By construction, each  nonconstant $1$-periodic orbit $x$ of $\hl$ is also a reparameterization of a closed 
orbit of a Hamiltonian on $T^*L$ of the form  $H(t,q,p)= \12 \|p\|^2 + V(t, q)$. Hence, one can associate to  $x$ an internal Conley-Zehnder index as well as a Morse index for its projection to $L$. Recall that $x$ lies either in 
 $U_{(r/5+\delta, 2r/5-\delta)}$ where $\nu$ is convex, or  $U_{(3r/5+\delta, 4r/5-\delta)}$ where
$\nu$ is concave. In the latter case, the identities above must be shifted in the following way.

\begin{Corollary}
\label{cor}
If  $x(t)=(q(t),p(t))$ is a $1$-periodic orbit of $\hl$ in  $U_{(r/5+\delta, 2r/5-\delta)}$, then
equations \eqref{orient} and \eqref{nonorient} hold. If $x$ is contained in $U_{(3r/5+\delta, 4r/5-\delta)}$,
then we have
\begin{equation}
\label{orient-b}
\czi(x) = \mor(q)-1.
\end{equation}
 if $q^*(TL)$ is trivial, and  
 \begin{equation}
\label{nonorient-b}
\czi(x) = \mor(q)-2
\end{equation}
otherwise.
\end{Corollary}

\begin{proof}
A proof of the shift in the concave case is contained in Proposition 2.1 of \cite{th}.
\end{proof}

Suppose that $x \in \PP(\hl)$ is contractible in $M$ and $w$ is a spanning disc for $x$.  One can  then define $\cz(x,w)$ as well as $\czi(x)$. Moreover,  $w$ determines a unique class in $\pi_2(M,L)$ and hence a Maslov index $\masl([w])$. These indices are related by the following identity which was first established by Viterbo in \cite{vi} for $(M,\om) = (\R^{2n}, \om_{2n})$.  We include a (different) proof, in the general case, for the sake of completeness. 

\begin{Proposition}
\label{in-out}
Let $x(t)=(q(t),p(t))$ be a nonconstant $1$-periodic orbit of $\hl$. Then for any spanning disc $w$ of $x$ we have 
\begin{equation*}
\label{ }
\cz(x,w) = \czi(x) + \masl([w]).
\end{equation*}
\end{Proposition}

\begin{proof}

Let 
$$
\Phi_w \colon S^1 \times \R^{2n} \to x^*(TT^*L)
$$
be a trivialization determined by the spanning disc $w$.
Let
$$
\Psi \colon S^1 \times \R^{2n} \to q^*(TT^*L).
$$
be a trivialization of the type necessary to compute $\czi(x)$.
That  is, $$
\czi(x) = \cz(\Psi(t)^{-1} \circ (d \phi^t_H)_{x(0)} \circ \Psi(0)).
$$
We may assume that
$$\Phi(0) = \Psi(0).$$
Using the loop property of the Conley-Zehnder index and identity \eqref{index ident},  we get 
\begin{eqnarray*}
\cz(x,w) & = & \cz(\Phi_w(t)^{-1} \circ (d \phi^t_H)_{x(0)} \circ \Phi_w(0)) \\
{} & = & \cz(\Phi_w(t)^{-1} \circ \Psi(t) \circ \Psi(t)^{-1} \circ (d \phi^t_H)_{x(0)} \circ \Psi(0) \circ \Psi(0)^{-1} \circ\Phi_w(0)) \\
{} & = & \cz(\Phi_w(t)^{-1} \circ \Psi(t) \circ \Psi(t)^{-1} \circ (d \phi^t_H)_{x(0)} \circ \Psi(0)) \\
{} & = & \czi(x)+ 2m (\Phi_w(t)^{-1} \circ \Psi(t))\\
{} & = & \czi(x)+\mas (\Gamma_{\Phi_w(t)^{-1} \circ \Psi(t)}).
\end{eqnarray*}
Since $\mas(\Phi_w(t)^{-1}(Vert(q(t))) = \masl([w]),$
it remains to prove that 
\begin{equation*}
\label{index identity}
\mas (\Gamma_{\Phi_w(t)^{-1} \circ \Psi(t)}) =  \mas(\Phi_w(t)^{-1}(Vert(q(t))).
\end{equation*}

Choose $V_0= \{0\}\times \R^n \subset \R^{2n}$ and note that for the trivialization $\Psi$ we have 
$\Psi(t)(V_0) = Vert(q(t)).$ Hence, it suffices to show that 
\begin{equation}
\label{index identity}
\mas (\Gamma_{\Phi_w(t)^{-1} \circ \Psi(t)}) =  \mas(\Phi_w(t)^{-1} \circ \Psi(t)(V_0)).
\end{equation}
Using the recipe for the Maslov index described in \S \ref{maslov}, we will verify \eqref{index identity} by proving that 
\begin{equation*}
\label{}
\mas (\Gamma_{\Phi_w(t)^{-1} \circ \Psi(t)}; V_0 \times V_0) =  \mas(\Phi_w(t)^{-1} \circ \Psi(t)(V_0); V_0).
\end{equation*}

By homotoping the trivializations $\Phi_w$ and $\Psi$, if necessary, we may assume that all of the relevant crossings are regular. 
Note that  $t$ is a crossing of $\Gamma_{\Phi_w(t)^{-1} \circ \Psi(t)}$ with respect to  $V_0 \times V_0$
if $$(v, \Phi_w(t)^{-1} \circ \Psi(t)(v)) \in V_0 \times V_0$$ for some nonzero $v \in \R^{2n}$. Similarly,  
$t$ is a crossing of $\Phi_w(t)^{-1} \circ \Psi(t)(V_0)$ with respect to $V_0$  
if there is some nonzero $v \in V_0$ such that 
$$\Phi_w(t)^{-1} \circ \Psi(t)(v) \in V_0.$$ Hence, the crossing are identical.

It remains to show that at each crossing the signatures of the relevant crossing forms are equal.
For simplicity we fix a crossing $t_0 $ and set $\Pi(t) =\Phi_w(t)^{-1} \circ \Psi(t)$. 
We also choose a $v \in V_0$ such that $\Pi(t_0)(v)$ is in $V_0$.

Fix a Lagrangian complement $W$ of $V_0$.
The first crossing form evaluated at $(v, \Pi(t_0)v)$ is
$$
Q(t_0)(v, \Pi(t_0)v) = -\om_{2n} \oplus \om_{2n}\left((v, \Pi(t_0)v), \dot{\widehat{w}}(t_0)\right)
$$
where $\widehat{w}(t)$ is a path in $W \times W$ defined by the condition 
\begin{equation}
\label{w}
(v, \Pi(t_0)v) + \widehat{w}(t) =(v(t), \Pi(t)v(t))
\end{equation}
for some path $v(t)$ in $\R^{2n}$ with $v(0)=v.$
From \eqref{w} we get 
\begin{eqnarray*}
Q(t_0)(v, \Pi(t_0)v) & = & -\om_{2n}(v, \dot{v}(0)) + \om_{2n}(\Pi(t_0)v, \dot{\Pi}(t_0)v) + \om_{2n}(\Pi(t_0)v, \Pi(t_0)\dot{v}(0)) \\
{} & = & \om_{2n}(\Pi(t_0)v, \dot{\Pi}(t_0)v). 
\end{eqnarray*}

Similarly, the second crossing form at $t_0$, when evaluated at $\Pi(t_0)v$, yields 
$$
Q(t_0)(\Pi(t_0)v)= \om_{2n}(\Pi(t_0)v, \dot{w}(0)).
$$
Here, $w(t)$ is a path in $W$ defined, for $t$ near zero, by the condition 
$$
\Pi(t_0)v + w(t) = \Pi(t)v(t),
$$
and $v(t)$ is now a path in $V_0$. In this case
\begin{eqnarray*}
Q(t_0)(\Pi(t_0)v) & = & \om_{2n}(\Pi(t_0)v, \dot{\Pi}(t_0)v) + \om_{2n}(\Pi(t_0)v, \Pi(t_0)\dot{v}(0)) \\
{} & = & \om_{2n}(\Pi(t_0)v, \dot{\Pi}(t_0)v)
\end{eqnarray*}
since both $\Pi(t_0)v$ and  $\Pi(t_0)\dot{v}(0)$ belong to the Lagrangian subspace $\Pi(t_0) (V_0).$ Clearly, the isomorphism of domains $(v,\Pi(t_0)v) \mapsto \Pi(t_0)v$ 
takes the first crossing form to the second and so they have the same signature, as desired. 
\end{proof}

\subsection{Cumulative bounds on the Maslov class}
Let $x(t)=(q(t),p(t))$ be a nonconstant $1$-periodic orbit of $\hl$, and let $w$ be a spanning disc for $x$. Using the relations above, we now obtain bounds on $\masl([w])$.
By Corollary \ref{cor} and Proposition \ref{in-out} we have 
\begin{equation*}
\label{ }
\masl([w]) =   \cz(z,w) -  \mor(q) (+1)(+1)
\end{equation*}
where the first  $(+1)$ contributes only if $q^*TM$ is  not orientable and the second $(+1)$
contributes if $x$ is contained in $U_{(3r/5+\delta, 4r/5-\delta)}$.

Applying Lemma \ref{potential} to Example \ref{split} we see that $q$ is a closed perturbed geodesic
whose Morse index satisfies
\begin{equation}
\label{basic}
\mor(q) \in [0, 1+\dim P_2 + \dots + \dim P_k].
\end{equation}
Overall, we then have 
\begin{Proposition}
\label{index}
Let $x(t) = (q(t), p(t))$ be a contractible $1$-periodic orbits of $\hl$ and let $w$ be a spanning disc for $x$.
Then 
\begin{equation}
\cz(x,w) -1-\sum_{j=2}^k \dim P_j  \leq \masl([w]) \leq \cz(x,w) (+1)(+1),
\end{equation}
where the first  $(+1)$ contributes only if $q^*TM$ is not orientable and the second $(+1)$
contributes if $x$ is contained in $U_{(3r/5+\delta, 4r/5-\delta)}$.
\end{Proposition}


\section{Hamiltonian paths which are not length minimizing and the proof of Theorem \ref{thm} }

Theorem \ref{thm}  follows from the index inequalities of Proposition \ref{index}
and the fact, established in Lemma \ref{short-g}, that $H_L$ does not minimize the negative Hofer length 
functional. The missing ingredient, which we describe in this section and prove in the next, is 
 a theorem which relates  the failure of an admissible Hamiltonian path to minimize  
the negative Hofer length to its $1$-periodic orbits.

Let $\JJ$ be the space of smooth compatible almost complex structures on $(M, \om)$.
When $(M,\om)$ is compact and convex, we say that $J \in \JJ$ is {\bf admissible} if 
\newcounter{Jcount}
  \begin{list}{{\bf (J\arabic{Jcount})}}
    {\usecounter{Jcount}
    \setlength{\rightmargin}{\leftmargin}}
   \item $\om (X(x),J(x) v) = 0$ for all $x \in \p M$ and $v \in T_x\p M$; 
  \item $\om (X(x), J(x)X(x)) =1$ for all $x \in \p M$;
  \item $\LL_X J=0$ on $M_{\eps}$ for some $\eps>0$.
  \end{list}  
If $(M,\om)$ is closed then every $J \in \JJ$ is admissible. In either case, the space of admissible almost complex structures will be denoted by $\widehat{\JJ}$.

For $J$ in $\JJ$, let $\hbar(J)$ be the infimum
over the symplectic areas of nonconstant $J$-holomorphic 
spheres in $M$. Set
$$
\widehat{\hbar} = \sup_{J \in \widehat{\JJ}} \hbar(J).
$$
The constant $\widehat{\hbar}$ is strictly positive and is greater than or equal to $r(M,\om)$.

\begin{Theorem}
\label{cap}
Let $(M,\om)$ be a symplectic manifold of dimension $2n$ which is 
either closed or compact and convex.
Let $H$ be an admissible Floer Hamiltonian on $M$, such that
 $\|H\| < \widehat{\hbar}$. If
$\phi^t_H$ does not minimize the negative Hofer seminorm in its 
homotopy class, then there is a $1$-periodic orbit $x$ of $H$ 
which admits a spanning disc $w$ such that 
$$ \cz(x,w) = n$$
and  
$$-\|H\|^- < \AC_H(x,w) \leq \|H\|^+.$$ 
\end{Theorem}
When $(M,\om)$ is closed, this result follows immediately
 from the main theorem of \cite{ke1}. In the next section, we present the 
proof for the case when $(M, \om)$ is convex. This proof also works in the closed case
but yields weaker results than those in \cite{ke1}.   

\subsection{Proof of Theorem \ref{thm} assuming Theorem \ref{cap}}

Before proving Theorem \ref{cap}, we first show that it implies Theorem \ref{thm}. In Lemma \ref{short-g} we proved  that $\phi^t_{\hl}$
does not minimize the negative Hofer length in its homotopy class. By construction, we also have 
$$
\|\hl\| < r(M,\om) \leq \widehat{\hbar}.
$$ 
Hence, Theorem \ref{cap} implies that there is a $1$-periodic orbit $x$ of $\hl$ which admits a spanning disc
$w$ such that $$ \cz(x,w) = n$$
and  
$$-\|\hl\|^- < \AC_H(x,w) \leq \|\hl\|^+.$$
It follows from Lemma \ref{nonconstant} that $x$ must be nonconstant. By Proposition \ref{index}, the Maslov index of the class $[w]$ then satisfies 
\begin{equation}
\label{end}
\dim P_1 -1 \leq \masl([w]) \leq n (+1) (+1).
\end{equation}
The lower bound $\dim P_1 -1$ is greater than zero since the dimension of $P_1$ must be at least two in order for it to admit a metric 
of negative sectional curvature. The upper bound is at most 
$n+1$ if $L$ is orientable and at most $n+2$ otherwise. Hence, inequality \eqref{end} implies the desired bounds
on $N_L$.

\section{Proof of Theorem \ref{cap}}

In this section we prove Theorem \ref{cap} under the assumption that  $(M,\om)$ is compact and convex.

\subsection{Overview}

Let $f \colon M \to \R$ be a Morse function on $M$ which is admissible in the sense of Section \ref{hamiltonian}. Fix a metric $h$ on $M$ so that  
the Morse complex $(\CM_*(f), \p_h)$  is well-defined.
Here, $\CM_*(f)$ is the vector space over $\Z_2$ which is generated by the critical points of $f$ and is graded by the Morse index.
The boundary map
$\p_h$ is defined by counting solutions of the negative gradient equation
\begin{equation}
\label{grad}
\dot{\gamma}= -\nabla_h f(\gamma).
\end{equation}
More precisely, $\p_h$, counts, modulo two, the elements of the spaces
$$
m(p,q)/\R \eqdef \{ \gamma \colon \R \to M \mid \dot{\gamma} = - \nabla_h f(\gamma) , \gamma(-\infty)=p,\, \gamma(+\infty) =q \}/\R,
$$
where $p$ and $q$ are critical points of $f$ with $\mor(p) = \mor(q) +1$, and $\R$ acts by translation on the argument of $\gamma$.
The homology of the Morse complex
is independent of both the admissible Morse function $f$ and the metric $h$,  and is isomorphic to $\H_*(M, \p M:\Z_2)$.

For a Floer Hamiltonian $H$ we will define a chain map
$$
\Phi \colon \CM(f) \to \CM(f)
$$
which counts rigid configurations that consist of solutions of \eqref{grad} and perturbed holomorphic cylinders which are asymptotic at one end to elements of $\PP(H)$.
If $\|H\| < \widehat{\hbar}$, then one can prove that $\Phi$ is chain homotopic to the identity.
The fact that the Morse homology is $\Z_2$ will then  yield an element $x \in \PP(H)$ which contributes to one of 
the configurations counted by $\Phi$. This will be the desired periodic orbit of $H$.  In particular, if the path $\phi^t_H$ does not minimize the negative Hofer length in its homotopy class, then this orbit will admit a spanning disc with respect to which its Conley-Zehnder index is equal to $\12 \dim M$, and its action satisfies the required bounds.

In the next five sections we recall the relevant Floer theoretic tools following the presentation of \cite{ke1}. The proof of Theorem \ref{cap} is then  contained in Section \ref{finally}.

\subsection{Homotopy triples and curvature} Let $\widehat{\JJ}_{S^1}$ be the space of smooth $S^1$-families of admissible almost complex structures in $\widehat{\JJ}$.
A smooth $\R$-family of Hamiltonians $F_s$ in $C^{\infty}(S^1 \times M)$ or 
elements in $\widehat{\JJ}_{S^1}$
will be called a {\bf compact homotopy} from $F^-$ to $F^+$, if there is an $\eta>0$ such that 
$$
F_s =
\left\{
  \begin{array}{ll}
    F^- ,& \hbox{ for $s \leq  -\eta$ ;} \\
    F^+, & \hbox{ for $s \geq  \eta$ .}
  \end{array}
\right.
$$ 
A compact homotopy $H_s$ of Hamiltonians with $H^-=0$ is called {\bf admissible} if
for some some $\eps>0$ and all $s > - \eta$ we have 
\begin{equation}
\label{hs}
H_s|_{S^1 \times M_{\eps}}(t,x,\tau) = a(s)e^{-\tau} +b(s,t)
\end{equation}
with $$a(s)<0$$ and $$\frac{d a}{ds} \leq 0.$$ 

Fix an admissible Floer Hamiltonian $H$ and an admissible family of almost complex structures $J$ in $\widehat{\JJ}_{S^1}$.
A {\bf homotopy triple} for the pair $(H,J)$ is a collection of compact homotopies 
$$\HH = (H_s,K_s, J_s),$$ such that
\begin{itemize}
\item $H_s$ is an admissible compact homotopy from the zero function to $H$;
\item $K_s$ is a compact homotopy of Hamiltonians in $C^{\infty}_0(S^1\times M)$ from the zero function to itself; 
\item $J_s$ is a compact homotopy  in $\widehat{\JJ}_{S^1}$ from some
$J^-$ to $J$.  
\end{itemize}
The {\bf curvature} of the homotopy triple $\HH=(H_s,K_s,J_s)$ is the function 
 $\kappa(\HH) \colon \R \times S^1 \times M \to \R$ defined by
$$\kappa(\HH) = \p_sH_s - \p_tK_s + \{H_s,K_s\}.$$
The positive and negative norms of the curvature are, respectively, 
$$
|||\kappa(\HH)|||^+ = \int_{\R \times S^1} \max_{p\in
M}\kappa(\HH)(s,t,p)\,\,ds \,dt,
$$
and 
$$
|||\kappa(\HH)|||^- = -\int_{\R \times S^1}  \min_{p \in
M}\kappa(\HH)(s,t,p) \,\,ds\,dt.
$$

\subsection{Floer caps} Given a homotopy triple $\HH=(H_s,K_s,J_s)$ for $(H,J)$, we consider smooth maps 
$u \colon \R \times S^1  \to M$, which satisfy the equation
\begin{equation}\label{left-section}
    \partial_s u- X_{K_s}(u)+ J_s(u)(\partial_tu - X_{H_s}(u))=0.
\end{equation}
The energy of a solution $u$ of \eqref{left-section} is defined as
$$
  E(u) = \int_{\R \times S^1} \om(u) \Big( \p_su -X_{K_s}(u), J_s (\p_su -X_{K_s}(u))\Big) \,ds \, dt.
$$
If this energy is finite, then 
$$ u(+\infty) \eqdef \lim_{s \to \infty} u(s,t) = x(t) \in
\PP(H)$$
and
$$
u(-\infty) \eqdef \lim_{s \to -\infty} u(s,t) = p \in M,
$$
where the convergence is in $C^{\infty}(S^1,M)$ and the point $p$ in $M$
is identified with the constant map $t \mapsto p$. This asymptotic behavior implies that if
a solution $u$ of \eqref{left-section} has finite energy,
then it determines an asymptotic spanning disc for the $1$-periodic orbit $u(+\infty)=x$. 
More precisely, for sufficiently large $s>0$, one can complete and reparameterize $u|_{[-s,s]}$ to be  
a spanning disc for $x$ in a homotopy class  which is independent of $s$. 

The set of {\bf left Floer caps} of $x \in \PP(H)$ with respect to $\HH$ is
$$\LL(x;\HH)= \left\{\ u \in C^{\infty}({\R \times S^1},M) \mid u
\text{ satisfies \eqref{left-section} },\,E(u)< \infty,\, u(+\infty)=x \right\}.
$$
For each $u \in \LL(x;\HH)$
we define the action of $x$ with
respect to $u$ by
$$
\AC_H(x,u) = \int_0^1 H(t,x(t))\,dt - \int_{\R \times S^1} u^*\om.
$$
A straight forward computation yields  
\begin{equation}\label{energy-left}
    0 \leq E(u) =   - \AC_H(x,u)+
\int_{\R \times S^1} \kappa(\HH)(s,t,u(s,t))\,ds \, dt.
\end{equation}
Each left Floer cap $u \in \LL(x;\HH)$ also determines a unique homotopy class of  trivializations of $x^*(T^*M)$ and
hence a Conley-Zehnder index $\cz(x,u)$.

For any function of the form $F(s,\cdot)$, we set  $$\overleftarrow{F}(s, \cdot)= F(-s, \cdot).$$
Given a homotopy triple $\HH=(H_s,K_s,J_s)$, we will also consider maps $v \colon
{\R \times S^1} \to M$ which satisfy the equation
\begin{equation}\label{right-section}
    \partial_sv + X_{\overleftarrow{K_s}}(v)+ \overleftarrow{J_s}(v)(\partial_tv - X_{\overleftarrow{H_s}}(v))=0.
\end{equation}
In this way, we obtain for each $x \in \PP(H)$ the space of {\bf right
Floer caps},
$$\RR(x;\HH)=\left\{\ v \in C^{\infty}({\R \times S^1},M) \mid v \text{ satisfies \eqref{right-section}} ,\,E(v)< \infty,\,
v(-\infty)=x \right\}.$$ 

Every right Floer cap $v \in \RR(x;\HH)$ also determines an asymptotic spanning disc for $x$,  
$$\overleftarrow{v}(s,t) =v(-s,t).$$  
The action of $x$ with respect to $\overleftarrow{v}$ is defined as 
$$
\AC_H(x,\overleftarrow{v}) = \int_0^1 H(t,x(t))\,dt - \int_{\R \times S^1} \overleftarrow{v}^*\om,
$$
and it satisfies the inequality
\begin{equation}\label{energy-right}
    0 \leq E(v) =  \AC_H(x,\overleftarrow{v})+
\int_{\R \times S^1} \kappa(\HH)(s,t,\overleftarrow{v}(s,t))\,ds \, dt.
\end{equation}
%
The Conley-Zehnder index of $x$ with respect to $\overleftarrow{v}$ is denoted by  
$\cz(x,\overleftarrow{v})$.

\subsection{Cap data and compactness}
\label{comp}

For the starting data $(H,J)$, we will choose  a pair of homotopy triples
$$
\Hh=(\HH_L,\HH_R).
$$
This will be referred to as our {\bf cap data}.
The norm of the curvature of $\Hh$ is defined to be
$$
|||\kappa(\Hh)||| = |||\kappa(\HH_R)|||^- + |||\kappa(\HH_L)|||^+.
$$
%

We will use $\Hh$ to define three classes of perturbed holomorphic cylinders. Two of these classes are the left Floer caps with respect to $\HH_L$, $\LL(x, \HH_L)$, and the right Floer caps with respect to $\HH_R$, $\RR(x, \HH_R)$. These are used to construct the map $\Phi$. The third class of perturbed holomorphic cylinders that we consider are called  {\bf Floer spheres}. These are defined  in Section \ref{phi} where they are used to construct the desired chain homotopy between $\Phi$ and the identity map. For each of the three classes of  perturbed holomorphic cylinders there are three possible sources of noncompactness. We need to avoid two of these sources, and in this section we describe how this is accomplished.

The first source  of noncompactness to be avoided is the possibility that a sequence of curves can approach the boundary of $M$. This possibility has already  been precluded by the admissibilty conditions on $\HH_L$ and $\HH_R$. Consider the case of left Floer caps.
For $\HH_L=(H_s,K_s,J_s)$, we note that  the admissibility conditions on $H_s$ and $J_s$, and the fact that $K_s$ belongs to $C^{\infty}_0(S^1 \times M)$ for all $s \in \R$, imply that for some $\eps>0$  equation \eqref{left-section} restricts to  $S^1 \times M_{\eps}$
as 
\begin{equation}\label{maximum}
    \partial_s u + J_s(u)(\partial_tu + a(s)R(u))=0.
\end{equation}
Here, each  $J_s$ satisfies conditions {\bf(J1)}-{\bf(J3)}. The function $a(s)$ is determined by $H_s$, as in \eqref{hs}, and so we have $a(s) \leq 0$ and $\frac{da}{ds} \leq 0$.
If $T \colon M_{\eps} \to \R$ is the function $T(\tau,x) = e^{\tau}$ and $u$ is a 
solution of \eqref{maximum}, then a straight forward computation yields 
\begin{equation*}
\label{ }
\triangle(T \circ u) = \om(\p_s u, J(u) \p_s u) - \frac{d a}{d s} (T \circ u),
\end{equation*}
(see, \cite{vi2} or Theorem 2.1 of \cite{fs}).
Since the right hand side is nonnegative, the Strong Maximum Principle implies that if  $T\circ u$ attains its maximum then it is constant, \cite{gt}. Hence, no left Floer cap $u \in \LL(x;\HH_L)$ enters 
$M_{\eps}$ and no sequence of left Floer caps can approach $\p M$.  The arguments for right Floer caps and the Floer spheres defined in Section \ref{phi}, are entirely similar and are left to the reader.

The other source of noncompactness that we wish to avoid is bubbling. 
To achieve this we will exploit the following fact; the energy of the Floer caps and Floer spheres
that we consider is bounded above by $||| \kappa (\Hh) |||$. Using the assumption from Theorem \ref{cap} that $\|H\| < \widehat{\hbar}$ we will construct
cap data $\Hh$ for which we have the curvature bound 
\begin{equation}
\label{curvature bound}
||| \kappa (\Hh) ||| \leq \|H\| < \widehat{\hbar}.
\end{equation}
As we now describe, this condition allows us to avoid bubbling  by simply restricting our choices of the almost complex structures which appear as part of the cap data $\Hh$. 

The following result is a simple consequence of Gromov's compactness theorem for 
holomorphic curves.
\begin{Lemma}(\cite{ke1})
\label{open}
For every $\delta>0$ there is a nonempty  open subset $\widehat{\JJ}^{\delta} \subset \widehat{\JJ}(M,\om)$ such that 
for every $J \in \widehat{\JJ}^{\delta}$ we have $\hbar(J) \geq \widehat{\hbar} - \delta$.
\end{Lemma}
Since $\|H\|< \widehat{\hbar}$, we can set
\begin{equation*}
\label{ }
\delta_H = \frac{\widehat{\hbar} -\|H\|}{2} 
\end{equation*} 
and let  $\widehat{\JJ}^{\delta_H}$ be an open set in $\widehat{\JJ}$ as described in Lemma \ref{open}. We will assume from now on that the almost complex structures which appear in the families in $\Hh$ all belong to  $\widehat{\JJ}^{\delta_H}$. This implies that any  that any bubble which forms from a sequence of Floer spheres,  must be a $J$-holomorphic sphere $w \colon S^2 \to M$ for some $J \in \widehat{\JJ}^{\delta_H}$. Inequality \eqref{curvature bound} implies that for this $J$ we have
$$
\hbar(J) \geq \widehat{\hbar} - \12 \left( \widehat{\hbar} -\|H\| \right) > ||| \kappa(\Hh)|||.
$$
As mentioned above,  the energy of our Floer spheres is bounded above by $|||\kappa(\Hh)|||$. 
This implies that the energy of the bubble $w$ is less than $|||\kappa(\Hh)|||$ and hence $\hbar(J)$. By the definitnion of $\hbar(J)$,  the bubble $w$ must therefore be trivial. Thus, if $\Hh$ satisfies \eqref{curvature bound} and if the families of almost complex structures appearing $\Hh$ take values in $\widehat{\JJ}^{\delta_H}$,  then no bubbling occurs for sequences 
of Floer caps or Floer spheres that are defined using $\Hh$.
 
We also note, that since  $\widehat{\JJ}^{\delta_H}$ is open we can achieve transversality for our spaces
of Floer caps and Floer spheres , which can therefore be assumed to be manifolds of the expected dimensions, \cite{fhs}.


\subsection{Specific cap data}

We now specify cap data $\Hh = (\HH_L, \HH_R)$ for $(H,J)$,  where $J$ takes values in
$\widehat{\JJ}^{\delta_H}$ and $H$ is a Hamiltonain as in Theorem \ref{cap}, i.e.,  $\|H\|< \widehat{\hbar}$ and $\phi^t_H$ does not minimize the negative Hofer length in its homotopy class.  This cap data will satisfy  the  curvature bound \eqref{curvature bound}. It will be assumed throughout, that the  families of almost complex structures appearing in $\Hh$ take values in $\widehat{\JJ}^{\delta_H}$.

Let $b \colon \R \to [0,1]$ be a smooth nondecreasing function 
such that $b(s) =0 $ for $s \leq -1$ and $b(s)=1$ for $s \geq 1$. Let $\HH_L$ be a \emph{linear} homotopy of the 
form
$$
\HH_L = (b(s)H, 0, J_{\scriptscriptstyle{L,s}}).
$$
For this choice, $\kappa(\HH_L) = \beta'(s)H$ and we have 
\begin{equation}
\label{k+}
||| \kappa(\HH_L)|||^+= \|H\|^+.
\end{equation}

Fix an admissible Hamiltonian $G$
such that $\phi^t_G$ is homotopic to $\phi^t_H$, relative its endpoints, and $\|G\|^- < \|H\|^-.$ 
Note that $H$ and $G$ have the same slope, $a_0$, because they generate the same time one map and are admissible.
 
We  now use the Hamiltonian $G$ to construct $\HH_R$.
We start with a linear homotopy triple for $G$ of the form 
$$
\GG = (b(s) G, 0,  J_s).
$$
Let $F_s$ be the normalized Hamiltonian which generates the Hamiltonian flow $\phi^t_{H} \circ (\phi^t_{b(s)H})^{-1} \circ \phi^t_{b(s)G}$. A straightforward computation shows that the slope of $F_s$ does not depend on $s$.  
Set
$$
\widetilde{H}_s = 
\left\{
  \begin{array}{lll}
    0 & \hbox{for $s \leq -1$;}\\  
    b(s)G & \hbox{for $-1 \leq s\leq 1 $;} \\
    F_{s-2}, & \hbox{for $1 \leq s \leq 3 $;}\\
    H, & \hbox{for $s \geq 3 $.}
   \end{array}
\right.
$$
This is an admissible compact homotopy from the zero function to $H$.
In particular, we have 
$$ \frac{da}{ds}(s) = \dot{b}(s) a_0 \leq 0,$$ where again $a_0$ is the common slope of $G$ and $H$.

Let $\varrho_t = \phi^t_{H} \circ (\phi^t_{G})^{-1}$.
The map $\Upsilon \colon C^{\infty}(S^1,M) \to C^{\infty}(S^1,M)$ defined by
\begin{equation}
\label{ups}
\Upsilon(x(t))=  \varrho_t ( x(t)).
\end{equation}
takes contractible loops to contractible loops and hence $\Upsilon (\PP(G))= \PP(H).$
Now consider the family of contractible Hamiltonian loops 
$$\varrho_{s,t}= \phi^t_{\widetilde{H}_s} \circ (\phi^t_{b(s)G})^{-1}.$$ 
For each value of $s$, $\varrho_{s,t}$ is a loop based at the identity,
and 
$$
\varrho_{s,t}= 
\left\{
  \begin{array}{ll}
    \id, & \hbox{for $s \leq 1$;} \\
    \phi^t_{F_{s-2}} \circ (\phi^t_{G})^{-1}, & \hbox{for $1 \leq s \leq 3$;} \\
   \phi^t_{H} \circ (\phi^t_{G})^{-1} , & \hbox{for $s \geq 3$.}
   \end{array}
\right.
$$
From $\varrho_{s,t}$ we obtain the family of normalized Hamiltonians, $A_s \in C^{\infty}_0(S^1 \times M)$,  defined by  
$$ \p_s (\varrho_{s,t}(p)) = X_{A_s}(\varrho_{s,t}(p)).$$
Set
$$
\HH_R = (\widetilde{H}_s,\widetilde{ K}_s, \widetilde{J}_{\scriptscriptstyle{R,s}} )
=(\widetilde{H}_s,  A_s , d  \varrho_{s,t} \circ J_s \circ d (\varrho_{s,t}^{-1})).$$
It is easy to verify that $\HH_R$ is a homotopy triple for $(H,J)$
where $J = \varrho_{t} \circ J^+ \circ d (\varrho_{t}^{-1}) $ and $J^+ =\lim_{s \to \infty } J_s$. It also follows from Proposition 2.7 of \cite{ke1} that 
\begin{equation}
\label{curve}
|||\kappa(\HH_R)|||^{\pm} = |||\kappa(\GG)|||^{\pm}.
\end{equation}

The following result is proved as Propositions 2.6 in \cite{ke1}.
\begin{Proposition}
\label{push}
The map $\widetilde{\Upsilon}$ defined on $\RR(x; \GG)$ by
$\widetilde{\Upsilon} (v(s,t))  = \varrho_{-s,t} (v(s,t)),$
is a bijection onto $\RR(  \Upsilon(x);\HH_R)$ for every $x \in \PP(G)$. and $$\cz(x,\overleftarrow{v}) = \cz(\Upsilon(x),\overleftarrow{\widetilde{\Upsilon}(v)})$$ and $$\AC_G(x,\overleftarrow{v}) = \AC_H(\Upsilon(x), \overleftarrow{\widetilde{\Upsilon}(v)}).$$
\end{Proposition}
The curvature identity \eqref{curve} yields
\begin{eqnarray*}
|||\kappa(\HH_R)|||^- &=& |||\kappa(\GG)|||^-\\
{} &=& -\int_{\R \times S^1} \dot{b}(s)(\min_{p \in M} G(t,p)) \,ds \,dt  \\
{} &=& \| G \|^-.
\end{eqnarray*}
Hence, by \eqref{k+} we have 
\begin{eqnarray*}
|||\kappa(\Hh)||| & = & |||\kappa(\HH_L)|||^+ + |||\kappa(\HH_R)|||^-\\
{} & = & \| H\|^+ + \|G\|^- \\
{} & < & \| H \| \\
{} & < & \widehat{\hbar}.
\end{eqnarray*}

\subsection{The map $\Phi$} 
\label{phi} 

Let $\Hh =(\HH_L, \HH_R)$ be the cap data for $(H, J)$ constructed in the 
previous section.
For each critical point $p$ of $f$ we consider the following two spaces of half-trajectories of the negative gradient equation \eqref{grad};
$$
\ell(p) =\left\{ \alpha \colon (-\infty,0] \to M \mid \dot{\alpha} = - \nabla_h f(\alpha),\, \alpha(-\infty)=p \right\}
$$
and 
$$
r(p) =\left\{ \beta \colon [0, +\infty) \to M \mid \dot{\beta} = - \nabla_h f(\beta),\, \beta(+\infty)=p \right\}.
$$
Let $\NN(p,x,q;f,\Hh)$ be the set of tuples $$(\alpha, u,v,\beta) \in \ell(p)\times \LL(x;\HH_L) \times \RR(x;\HH_R) \times r(q)$$ such that 
\begin{equation*}
\label{ }
 \alpha(0) = u(-\infty),
\end{equation*}
\begin{equation*}
\label{ }
v(+\infty) = \beta(0)
\end{equation*}
and
\begin{equation*}
\label{ }
[u \#v ] =0.
\end{equation*}
It follows from \cite{pss} that for generic choices of the families of almost complex structures appearing in $\Hh$, the 
dimension of $\NN(p,x,q;f,\Hh)$ is $\mor(p) -\mor(q)$.

The map $\Phi \colon \CM_*(f) \to \CM_*(f)$,  is then defined by setting the coefficient 
of $q$ in the image of $p$ to be the number of elements, modulo two, of the set
$$
\bigcup_{x \in \PP(H)} \NN(p,x,q;f, \Hh).
$$ 
To show that $\Phi$ is well-defined, we must verify that the zero-dimensional spaces $\NN(p,x,q;f, \Hh)$ are compact. The only
possible source of noncompactness are bubbles which appear on the Floer caps of the configurations in $\NN(p,x,q;f, \Hh).$ To avoid this, we only need to show that for each tuple $(\alpha,u,v,\beta)$ in $\NN(p,x,r;f,\Hh)$, both $u$ and $v$ have energy less than or equal to $|||\kappa(\Hh)|||$. Then, as described in Section \ref{comp}, our choice of almost complex structures in $\Hh$ precludes such bubbling. 

Equations \eqref{energy-left} and \eqref{energy-right} imply that 
\begin{equation*}
\label{left}
0 \leq E(u) \leq -\AC_H(x,u) + |||\kappa(\HH_L)|||^+
\end{equation*}
and 
\begin{equation*}
\label{right}
0 \leq E(v) \leq \AC_H(x,\overleftarrow{v}) - |||\kappa(\HH_R)|||^-.
\end{equation*}
Since $[u\#v]=0$, we have $\AC_H(x,u) =\AC_H(x,\overleftarrow{v})$ and 
\begin{eqnarray*}
E(u) & \leq  & -\AC_H(x,\overleftarrow{v}) + |||\kappa(\HH_L)|||^+ \\
{} & \leq  & |||\kappa(\HH_R)|||^- + |||\kappa(\HH_L)|||^+\\
{} & = & |||\kappa(\Hh)|||.
\end{eqnarray*}
A similar computation implies that  $E(v) \leq |||\kappa(\Hh)|||$, and so $\Phi$ is well-defined. 

\begin{Proposition}\label{prop:ident}
The map $\Phi \colon \CM(f) \to \CM(f)$ is chain homotopic to the identity.
\end{Proposition}

\begin{proof}

The specific cap data, $\Hh=(\HH_L,\HH_R)$, is used to construct a \emph{ homotopy of homotopy triples} $\HH^{\lambda} = (H_s^{\lambda}, K_s^{\lambda}, J_s^{\lambda})$ for $\lambda \in [0,+\infty)$. The desired chain homotopy is then defined using maps
$w$ in $\C^{\infty}({\R \times S^1},M)$ which satisfy the equation
\begin{equation}
\label{middle}
\p_s w- X_{K^{\lambda}_s}(w)+ J^{\lambda}_s(w)(\p_tw - X_{H^{\lambda}_s}(w))=0.
\end{equation}

To define $\HH^{\lambda}$ we introduce the notation $\HH = \HH(s)=(H_s,K_s,J_s)$ to emphasize the $s$-dependence of $\HH$. 
Set
 $$
\HH^{1}(s) =  \begin{cases}
   \HH_{L}(s+3)& \text{ when $ s \leq 0$}, \\
   \HH_{R}(s-3)& \text{ when $ s \geq 0$}.
   \end{cases}
$$
For $\lambda \in [1,+\infty)$, we then define 
$$
\HH^{\lambda}(s) =  \begin{cases}
   \HH_{L}(s+C(\lambda))& \text{ when $ s \leq 0$}, \\
   \HH_{R}(s-C(\lambda))& \text{ when $ s \geq 0$},
   \end{cases}
$$
where $C(\lambda)$ is a smooth nondecreasing function which equals $\lambda$ for $\lambda \gg 1$ and is 
equal to $3$ for $\lambda$ near $1$.  Finally, for $\lambda \in [0,1]$ we set 
$$\HH^{\lambda}(s)= (D(\lambda)H^1_s, D(\lambda)K^1_s, J^{\lambda}_s)$$ 
for a smooth nondecreasing function $D \colon [0,1] \to [0,1]$ which equals zero near 
$\lambda = 0$ and equals one near  $\lambda=1$. The compact homotopies 
$J^{\lambda}_s$ are chosen so that they equal $J^0_s$  for $\lambda$ near zero
and equal $J^1_s$ for $\lambda$ near one. 

The fact that  $J_s$ takes values in $\widehat{\JJ}^{\delta_H}$ implies the same for the families 
$J^{\lambda}_s$ with $\lambda \geq 1$. For $\lambda \in [0,1)$, we choose the families $J^{\lambda}_s$ so that they also take values in $\widehat{\JJ}^{\delta_H}$.

The following properties of $\HH^{\lambda}$ are easily verified:  
\begin{enumerate}
  \item For each $\lambda\in [0,+\infty)$, $H_s^{\lambda}$ is an admissible compact homotopy from the zero function to itself.
  \item For each $\lambda\in [0,+\infty)$, $K_s^{\lambda}$ is a compact homotopy in $C^{\infty}_0(S^1 \times M)$
  from the zero function to itself.
  \item There is an $\eps>0$ such that  equation \eqref{middle} restricts to  $S^1 \times M_{\eps}$
as 
\begin{equation}\label{max}
    \partial_s w + J^{\lambda}_s(w)(\partial_t w + a(\lambda,s)R(w))=0,
\end{equation}
where $a(\lambda,s)$ is determined by $H^{\lambda}_s|_{M_{\eps}}$ and satisfies $a(\lambda,s) \leq 0$ and $\frac{\p a}{\p s} \leq 0.$
\item $|||\kappa(\HH^{\lambda})|||^+  = D(\lambda)|||\kappa(\Hh)|||$ for all $\lambda \in [0,+\infty)$.
\end{enumerate}

Solutions of \eqref{middle} are perturbed holomorphic cylinders which are asymptotic, at both ends, to points in $M$.
In particular, since the perturbations are compact, each $w$ can be uniquely completed to a perturbed holomorphic sphere. We define the space of {\bf Floer spheres} for $\Hh$ as
$$\LL\RR(\Hh)=\left\{ (\lambda,w)  \in [0,+\infty) \times \C^{\infty}({\R \times S^1},M) \mid w \text{  satisfies  } \eqref{middle},\, [w]=0 \in \pi_2(M) \right\}.$$
Invoking again the Strong Maximum Principle, the third property of $\HH^{\lambda}$, listed above, implies that the distance between the images of the maps $w$  in $\LL\RR(\Hh)$ and the boundary of $M$ is bounded away from zero buy a positive constant which is independent of $\lambda$. 

For every $(\lambda,w) \in \LL\RR(\Hh)$, we also have the uniform energy bound
\begin{eqnarray*}
E(\lambda,w) & = &  \int_{\R \times S^1} \om \left( \p_s w -X_{K^{\lambda}_s}(w), J^{\lambda}_s (\p_sw -X_{K^{\lambda}_s}(w))\right) \,ds \, dt\\
& = &  \int_{\R \times S^1} \Big( \om(X_{H^{\lambda}_s}(w), \p_sw)-\om(X_{K^{\lambda}_s}(w), \p_tw)+\om(X_{K^{\lambda}_s}(w),X_{H^{\lambda}_s}(w)) \Big) \,ds \, dt\\
 & =&  \int_{\R \times S^1} \kappa(\HH^{\lambda})(s,t,w) ds \, dt \\
 &\leq& |||\kappa(\HH^{\lambda})|||^+\\
 &=& D(\lambda) ||| \kappa(\Hh) |||\\
 &<& \widehat{\hbar}. 
 \end{eqnarray*}
As described in Section \ref{comp}, this allows us to rule out the possibility of bubbling for sequences in $\LL\RR(\Hh)$.  

For a pair of critical points $p$ and $q$ of $f$, we define $\NN_{\lambda}(p,q)$ by
$$
\Big\{ \big(\alpha,(\lambda,w), \beta \big) \in \ell(p)  \times \LL\RR(\Hh) \times r(q)\mid
\alpha(0)= w(-\infty),\, w(+\infty)=\beta(0)\Big\}
$$ 
For generic data, $\NN_{\lambda}(p,q)$
is a manifold of dimension $\mor(p) - \mor(q)+1$, \cite{pss}.
We define the map $\chi \colon \CM_*(f) \to \CM_{*+1}(f)$ by 
setting
$$
\chi(p) =\sum_{\cz(q) = \cz(p)+1} \#_2 (\NN_{\lambda}(p,q))q
$$
where $ \#_2 (\NN_{\lambda}(p,q))$ is the number of components in the compact 
zero-dimensional manifold $\NN_{\lambda}(p,q)$, modulo two.
To prove that $\chi$ is the desired chain homotopy, it suffices 
to show that for every pair of critical points  $p$ and $r$, of $f$,  such
that $\mor(p)=\mor(r)$,  the coefficient
of $r$ in 
\begin{equation}
\label{r coeff}
\big(\id - \Phi + \chi \circ \p_h + \p_h \circ \chi\big)(p) 
\end{equation} 
is zero.

Consider the compactification $\overline{\NN}_{\lambda}(p,r)$ 
of the one dimensional moduli space $\NN_{\lambda}(p,r)$. 
Since the elements of $\NN_{\lambda}(p,r)$ can not approach the boundary of $M$ and we have precluded bubbling, it follows from Floer's
gluing and compactness theorems that the boundary
of $\overline{\NN}_{\lambda}(p,r)$  can be identified with the  
union of the following zero-dimensional manifolds:
\begin{enumerate}
\renewcommand{\theenumi}{\roman{enumi}}
\renewcommand{\labelenumi}{(\theenumi)}
  \item $m(p,r)$,\\ 
  \item $\bigcup_{x \in \PP(H)} \NN(p,x,r;f,\Hh),$\\ 
  \item $\bigcup_{\cz(q) = \cz(p)-1} m(p,q)/\R \times \NN_{\lambda}(q,r),$\\
  \item $\bigcup_{\cz(q) = \cz(p)+1} \NN_{\lambda}(p,q) \times m(q,r)/\R.$
\end{enumerate}

By definition, the number of elements in the sets (ii), (iii) and (iv), modulo  two, are the  
coefficients of $r$ in $\Phi(p)$, $\chi \circ \p_h(p)$ and $\p_h \circ\chi(p)$, respectively. 
The first set $m(p,r)$ is empty if $p \neq r$, and  consists only of the 
constant map when $p=r$. So, the number of elements in $m(p,r)$, modulo two, is equal to the 
coefficient of $r$ in $\id(p)$. As the boundary of the one-dimensional manifold $\NN_{\lambda}(p,r)$, the total number of these boundary terms is zero, modulo two. Hence, the coefficient of $r$ in \eqref{r coeff} is zero, and $\chi$ is the desired chain homotopy.
\end{proof}

\subsection{Completion of the proof of Theorem \ref{cap}}
\label{finally}

For simplicity, we 
assume that the admissible Morse function $f$ on $M$ has a unique local, and hence global, minimum
at a point $P$ in $M$. Since it is the only critical point of $f$ with the smallest possible Morse index, $P$ is a cycle in the Morse complex of $f$ and is the only representative 
of  $\H_0(\CM(f), \p_h) =\Z_2.$ By Proposition \ref{prop:ident} we then have $\Phi(P)=P$ at the level of chains. Hence, for some $x \in \PP(H)$ there is a  
a tuple $(\alpha, u, v,\beta)$ in the zero-dimensional space $\NN(P,x,P;f,H)$. 
The pairs $(\alpha, u)$ and $(v, \beta)$ each belong to moduli spaces with the same asymptotic behavior.  In fact, these moduli spaces are used in the definition of the Piunikhin-Salamon-Schwarz maps from \cite{pss}, where their dimensions are computed to be
$n-\cz(x,u)$ and $\cz(x,\overleftarrow{v}) -n$, respectively. 
Since  $\NN(P,x,P;f,H)$ has dimension zero, the moduli spaces containing $(\alpha, u)$ and $(v, \beta)$ must also be zero-dimensional and so \begin{equation}
\label{i}
\cz(x,u) = \cz(x, \overleftarrow{v}) = n.
\end{equation}

The condition that $[u\#v] = 0$ implies that
$$
\AC_H(x,u) =\AC_H(x, \overleftarrow{v}).
$$
By Proposition \ref{push} we have 
\begin{equation*}
\label{ }
\AC_H(x, \overleftarrow{v}) = \AC_G(\Upsilon^{-1}(x), \overleftarrow{\widetilde{\Upsilon}^{-1}(v)})
\end{equation*}
where $\widetilde{\Upsilon}^{-1}(v)$ is a right Floer cap in
$\RR(\Upsilon^{-1}(x); \GG)$ and $\GG$ is a linear homotopy
triple of the form $(\beta(s)G,0,J_s)$. Inequality \eqref{energy-right}
then implies that 
$$
\AC_H(x, \overleftarrow{v}) \geq - |||\kappa(\GG)|||^- = -\|G\|^- > -\|H\|^-.
$$
On the other hand, inequality \eqref{energy-left} yields
$$
\AC_H(x,u) \leq |||\kappa(\HH_L)|||^+ = \|H\|^+,
$$
since $\HH_L$ is the linear homotopy triple $(\beta(s)H, 0, J_{\scriptscriptstyle{L,s}})
$ and 
$$
|||\kappa(\HH_L)|||^+ = \int_{\R \times S^1} \dot{b}(s)\left( \max_{p\in M}H(t,p) \right) \,ds \,dt = \|H\|^+.
$$
Altogether, we have 
\begin{equation}
\label{a}
\|H\|^-  < \AC_H(x, \overleftarrow{v}) = \AC_H(x,u) \leq \|H\|^+.
\end{equation}

Let $w$ be a genuine spanning disc  for $x$ which is obtained by reparameterizing the asymptotic spanning disc $u$.
By \eqref{i} and \eqref{a}, we have $\cz(x,w)=n$ and  $\|H\|^- <  \AC_H(x,w)  \leq \|H\|^+$, as required.

\section{Generalizations}
We end this note by describing some simple generalizations of Theorem \ref{thm}.
To begin with, one expects that these rigidity results should hold for more general classes of ambient symplectic manifolds. Indeed, one can immediately extend Theorem \ref{thm} to products of convex symplectic manifolds or the more general class of \emph{split convex} manifolds defined in \cite{fs}.

Theorem \ref{thm} also holds for more general classes of Lagrangian submanifolds. In particular, the methods used here only detect those closed geodesics on a Lagrangian submanifold which are
contractible in the ambient symplectic manifold. Hence, the theorem also holds for easily displaceable Lagrangian 
submanifolds of the form $L=L_1 \times L_2,$ where $L_1$ is split hyperbolic and 
$L_2$ admits a metric which has no nonconstant contractible geodesics and is incompressible
in the sense that the map $\pi_1(L_2) \to \pi_1(M),$ induced by inclusion, is an injection. 

Combining these observations, we get

\begin{Corollary}
Suppose that the symplectic manifold $(M,\om)$ is rational, proportional and either closed or split convex.
Let $L=L_1 \times L_2$ be an easily displaceable Lagrangian submanifold of $(M,\om)$
such that  $L_1$ is split hyperbolic and 
$L_2$  is incompressible and admits a metric which has no nonconstant contractible geodesics. Then $N_L \leq \12 \dim M+2,$ and if 
$L$ is orientable we have $N_L \leq \12 \dim M+1.$
\end{Corollary}

\end{document}